\documentclass[12pt,reqno]{amsart}
\theoremstyle{plain}
\usepackage{graphicx,amssymb,mathrsfs}
\usepackage{amsmath}
\usepackage{amssymb}
\usepackage{amsthm}
\usepackage{verbatim,fullpage}
\usepackage{amsfonts}
\usepackage{graphicx}
\usepackage{xspace,url,hyperref}
\usepackage[nohug]{diagrams}
\renewcommand{\le}{\leqslant}
\renewcommand{\ge}{\geqslant}
\renewcommand{\leq}{\leqslant}
\renewcommand{\geq}{\geqslant}

\newcommand{\e}{\varepsilon}

\newtheorem{corollary}{Corollary}
\newtheorem{theorem}[corollary]{Theorem}
\newtheorem{lemma}[corollary]{Lemma}
\newtheorem{claim}[corollary]{Claim}

\newtheorem{proposition}[corollary]{Proposition}
\newtheorem{question}{Question}
\newtheorem{remark}{Remark}
\begin{document}

%\title{Politics: A game theoretic construction of optimal Lipschitz extensions on trees}

\title{Absolutely minimal Lipschitz extension \\ of tree-valued
mappings}\thanks{A.~N. is supported by NSF grants CCF-0635078 and
CCF-0832795, BSF grant 2006009, and the Packard Foundation. S.~S. is
supported by NSF grants DMS-0645585 and OISE-0730136.}

%\author{Assaf Naor\\ Courant Institute\\{\tt naor@cims.nyu.edu} \and Scott Sheffield\\ Massachusetts Institute of Technology\\{\tt sheffield@math.mit.edu}}

\author{Assaf Naor}
\address{Courant Institute\\ New York University}
\email{naor@cims.nyu.edu}

\author{Scott Sheffield}
\address{Massachusetts Institute of Technology}
\email{sheffield@math.mit.edu}

\date{}
\maketitle
\font \m=msbm10
\newcommand{\R}{{{\mathbb R}}}
\newcommand{\C}{{{\mathbb C}}}
\newcommand{\Z}{{ {\mathbb Z}}}
\newcommand{\N}{{\hbox {\m N}}}
\newcommand{\U}{{\hbox {\m U}}}
\newcommand{\HH}{{\hbox {\m H}}}
\renewcommand{\N}{\mathbb N}
\renewcommand{\epsilon}{\varepsilon}

\newcommand{\old}[1]{}
\newcommand{\thh}{\ensuremath{^{\text{th}}}\xspace}
\newcommand{\nd}{\ensuremath{^{\text{nd}}}\xspace}
\newcommand{\st}{\ensuremath{^{\text{st}}}\xspace}
\newcommand{\G}{{G}}
\newcommand{\E}{{\mathbb E}}
\newcommand{\Eng}{{\cal E}}
\newcommand{\M}{{\cal M}}
\newcommand{\T}{{\mathbb T}}
\newcommand{\supp}{{\operatorname{supp}}}
\newcommand{\sign}{{\operatorname{sign}}}
\newcommand{\Arg}{{\operatorname{Arg}}}
\newcommand{\Log}{{\operatorname{Log}}}
\newcommand{\Lip}{{\operatorname{Lip}}}
\newcommand{\diam}{{\operatorname{diam}}}
\newcommand{\A}{{\mathbb A}}
\newcommand{\I}{{\rm I}\xspace}
\newcommand{\II}{{\rm II}\xspace}
\newcommand{\favored}{{\mathfrak F}}
\newcommand{\eps}{\varepsilon}
\newcommand{\Cov}{{\rm Cov}}
\newcommand{\Var}{{\bf Var}}
\renewcommand{\Re}{\operatorname{Re}}
\renewcommand{\Im}{\operatorname{Im}}
\newcommand{\Xt}{X_{\text{term}}}
\renewcommand{\phi}{\varphi}
\newcommand{\increasing}{\mbox{$\star$-increasing}\xspace}
\newcommand{\decreasing}{\mbox{$\star$-decreasing}\xspace}
\renewcommand{\setminus}{\smallsetminus}
\renewcommand{\vec}{}
\newcommand{\eqdef}{\stackrel{\mathrm{def}}{=}}

\newarrow{Dotsto} ....>
\newarrow{Dashto} {}{dash}{}{dash}>
\newarrow {Backwards} <----

\newcommand{\vs}{\vspace{.15 in}}

\def \interior {{\hbox {int}}}
\def \eps {\varepsilon}
\def \Lip {\mathrm{Lip}}

\begin {abstract}
We prove that every Lipschitz function from a subset of a locally
compact length space to a  metric tree has a unique absolutely
minimal Lipschitz extension (AMLE).  We relate these extensions to a
stochastic game called {\bf Politics} --- a generalization of a game
called {\bf Tug of War} that has been used in~\cite{PSSW09} to study
real-valued AMLEs.
\end {abstract}

%\tableofcontents

\section {Introduction}

For a pair of metric spaces $(X,d_X)$ and $(Z,d_Z)$, a mapping
$h:X\to Z$, and a subset $S\subseteq X$, the Lipschitz constant of
$h$ on $S$ is denoted
$$
\Lip_S(h) \eqdef \sup_{\substack{x,y \in S\\x\neq y}}
\frac{d_Z(h(x),h(y))}{d_X(x,y)}.$$ Given a closed subset $Y\subseteq
X$ and a Lipschitz mapping $f:Y\to Z$, a Lipschitz mapping
$\widetilde f:X\to Z$ is called an {\bf absolutely minimal Lipschitz
extension} (AMLE) of $f$ if its restriction to $Y$ coincides with
$f$, and for every open subset $U\subseteq X\setminus Y$ and every
Lipschitz mapping $h:X\to Z$ that coincides with $\widetilde f$ on
$X\setminus U$ we have \begin{equation}\label{eq:modification}
\Lip_U(h)\ge \Lip_{U}\left(\widetilde f\right).
\end{equation}
 In other words, $\widetilde f$ extends $f$, and it is not possible to modify
 $\widetilde f$ on an open set in a way that decreases the Lipschitz
constant on that set.

Our main result is:
\begin{theorem}\label{thm:tree main}
Let $X$ be a locally compact length space and let  $T$ be a metric tree. For every closed subset $Y\subseteq X$, every Lipschitz mapping $f:Y\to T$ has a unique AMLE $\widetilde f:X\to T$.
\end{theorem}

Recall that a metric space $(X,d_X)$ is a \textbf{length space} if
for all $x,y\in X$, the distance $d_X(x,y)$ is the infimum of the
lengths of curves in $X$ that connect $x$ to $y$. By a
\textbf{metric tree} we mean the one-dimensional simplicial complex
associated to a finite graph-theoretical tree with arbitrary edge
lengths (i.e., a finite graph-theoretical tree whose edges are
present as actual intervals of arbitrary length, equipped with the
graphical shortest path metric). We did not investigate here the
greatest possible generality in which Theorem~\ref{thm:tree main}
holds true; in particular, we conjecture that the assumption that
$X$ is locally compact can be dropped, and that $T$ need not
correspond to a finite graph-theoretical tree, but rather can belong
to the more general class of bounded $\R$-trees
(see~\cite{MR753872,MR1379369}). The requirement that $X$ be locally
compact is not used in our proof of the uniqueness assertion of
Theorem~\ref{thm:tree main}.

In the special case when $T$ is an interval $[a,b]\subseteq \R$ and
$X=\R^n$, Theorem~\ref{thm:tree main} was proved in~\cite{Jen93};
see also~\cite{MR1876420,ACJ04,AS10} for different proofs of the
uniqueness part of Theorem~\ref{thm:tree main} in this special case.
The existence part of Theorem~\ref{thm:tree main} was generalized to
arbitrary length spaces $X$ and $T=[a,b]$ in~\cite{MR1719573}; see
also~\cite{Juu02} and~\cite{MR2346452} for different proofs of this
existence result with the additional assumptions that the length
space $X$ is separable or compact, respectively. The uniqueness part of Theorem~\ref{thm:tree main} was proved in~\cite{PSSW09} for
$X$ a general length space and $T=[a,b]$. Additionally, \cite{PSSW09} contains a new (game theoretic) proof of the existence part of Theorem~\ref{thm:tree main} when $T=[a,b]$ and $X$ is a general length space.

The purpose of the present article is to initiate the study of
absolutely minimal Lipschitz extensions of mappings that are not
necessarily real-valued, the tree-valued case being the first
non-trivial setting of this type where such theorems can be proved.
Our proofs overcome various difficulties that arise since we can no
longer use the order structure of the real line, which was crucially
used in~\cite{Jen93,MR1719573,Juu02,MR2346452,PSSW09}.  We also introduce
a stochastic game called {\bf Politics}, related to tree-valued AMLE,
that generalizes the stochastic game called {\bf Tug of War} that was introduced
and related to real-valued AMLE in \cite{PSSW09}.

\begin{comment}
In the special case when $T$ is an interval $[a,b]\subseteq \R$ and
$X=\R^n$, Theorem~\ref{thm:tree main} was proved in~\cite{Jen93}
(see also~\cite{MR1876420,ACJ04} for different proofs of the
uniqueness statement in Theorem~\ref{thm:tree main}). The existence
part of Theorem~\ref{thm:tree main} was generalized to $X$ a
separable length space in~\cite{MR1719573,Juu02} (see
also~\cite{MR2346452} for a different proof when $X$ is a compact
length space), and was further generalized to length spaces that are
not necessarily separable in~\cite{PSSW09}, where uniqueness was
first established in this generality as well. The main purpose of
the present paper is to initiate the study of absolutely minimal
Lipschitz extension for mappings that are not necessarily
real-valued, the tree-valued case being the first non-trivial
setting of this type where such theorems can be proved. Our proofs
overcome various difficulties that arise since we can no longer use
the order structure of the real line, which was crucially used
in~\cite{Jen93,MR1719573,Juu02,MR2346452,PSSW09}.
\end{comment}

In the remainder of this introduction we explain the relevant
background from the classical theory of Lipschitz extension and
$\infty$-harmonic functions, and also describe the main steps of our
proof.

%We end with some natural open problems.

%, and a discussion of some potential applications of our results.

%we have  $\Lip_U\left(\widetilde f\right)=\Lip_{\partial U}\left(\widetilde f\right)$

\subsection{Background on the Lipschitz extension problem}

The classical Lipschitz extension problem asks for conditions on a
pair of metric spaces $(X,d_X)$ and $(Z,d_Z)$ which ensure that
there exists $K\in (0,\infty)$ such that for all $Y\subseteq X$ and
all Lipschitz mappings $f:Y\to Z$, there exists $\widetilde f:X\to
Z$ with $\left. \widetilde f\right|_Y=f$ and
\begin{equation}\label{eq:ext condition}
\Lip_X\left(\widetilde f\right)\le K\cdot \Lip_Y(f).
\end{equation}
%Here, and in what follows, for each $h:X\to Z$ and $U\subseteq X$ we write $\Lip_U(h)$ for the Lipschitz constant of the restriction of $h$ to $U$, i.e.,
%$$
%\Lip_U(h) \eqdef \sup_{\substack{x,y \in U\\x\neq y}} \frac{d_Z(h(x),h(y))}{d_X(x,y)}.$$

Stated differently, in the Lipschitz extension problem we are interested in geometric conditions ensuring the existence of $\widetilde f:X\to Z$ such that the diagram in~\eqref{eq:diagram} commutes, where $\iota:Y\to X$ is the formal inclusion, and the Lipschitz constant of $\widetilde f$ is guaranteed to be at most a fixed multiple (depending only on the geometry of the spaces $X, Z$) of the Lipschitz constant of $f$.
\begin{equation}\label{eq:diagram}
\begin{diagram}
X & & \\
\dBackwards^\iota &  \rdDashto^{\widetilde f} & \\
Y& \rTo^{f} & Z
\end{diagram}
\end{equation}
Note that if $(Z,d_Z)$ is complete then we can trivially extend $f$ to the closure of $Y$.

%We will therefore assume throughout  that $(Z,d_Z)$ is complete and $Y$ is closed.

When $K=1$ in~\eqref{eq:ext condition}, i.e., when one can always
extend functions while preserving their Lipschitz constant, the pair
$(X,Z)$ is said to have the {\bf isometric extension property}. When
$K\in (1,\infty)$ the corresponding extension property is called the
{\bf isomorphic extension property}. The present article is devoted
to the isometric extension problem, though we will briefly discuss
questions related to its isomorphic counterpart in
Section~\ref{sec:future}. We refer to the
books~\cite{MR0461107,BL00} and the references therein, as well as
the introductions of~\cite{LN05,NPSS06} (and the references
therein),  for more background on the Lipschitz extension problem.

It is rare for a pair of metric spaces $(X,Z)$ to have the isometric
extension property. A famous instance when this does happen is
Kirszbraun's extension theorem~\cite{Kirsz34}, which asserts that if
$X$ and $Z$ are Hilbert spaces then $(X,Z)$ have the isometric
extension property. Another famous example is the non-linear
Hahn-Banach theorem~\cite{McShane34}, i.e., when $Z=\R$ and $X$ is
arbitrary; this (easy) fact follows from the same proof as the proof
of the classical Hahn-Banach theorem (i.e., by extending to one
additional point at a time; alternatively, one can construct the maximal
and minimal isometric extensions explicitly).

More generally, one may consider  metric spaces $Z$ such that for
every metric space $X$ the pair $(X,Z)$ has the isometric extension
property (i.e., $Z$ is an injective metric space in the isometric
category). This is equivalent to the fact that there is a
$1$-Lipschitz retraction from any metric space containing $Z$ onto
$Z$ (see~\cite[Prop.~1.2]{BL00}); such spaces are called in the
literature {\bf absolute $1$-Lipschitz retracts}. It is a well known
fact (see~\cite[Prop.~1.4]{BL00}) that $(Z,d_Z)$ is an absolute
$1$-Lipschitz retract if and only if $(a)$ $Z$ is {\em metrically
convex}, i.e., for every $x,y\in Z$ and $\lambda\in [0,1]$ there is $z\in
Z$ such that $d_Z(x,z)=\lambda d_Z(x,y)$ and $d_Z(y,z)=(1-\lambda)d_Z(x,y)$, and
$(b)$ $Z$ has the {\em binary intersection property}, i.e., if every
collection of pairwise intersecting closed balls in $Z$ has a common
point. Examples of absolute $1$-Lipschitz retracts are $\ell_\infty$
and metric trees (see~\cite{KN81,JLPS02}). Additional examples are
contained in~\cite{Isb64} (see also~\cite[Ch.~1]{BL00}).

If $(X,d_X)$ is path-connected and the pair $(X,Z)$ has the
isometric extension property, then the AMLE
condition~\eqref{eq:modification} is equivalent to the requirement:
 %that for every open $U\subseteq X\setminus Y$ we have
\begin{equation}\label{eq:def boundary}
\forall \ \mathrm{open}\ U\subseteq X\setminus Y,\quad \Lip_{U}\left(\widetilde f\right)=\Lip_{\partial U}\left(\widetilde f\right).
\end{equation}
When $Z=\R$ and $X=\R^n$, the AMLE formulation~\eqref{eq:def
boundary} was first introduced by Aronsson~\cite{Aro67}, in
connection with the theory of $\infty$-harmonic functions.
Specifically, it was shown in~\cite{Aro67} that if $\widetilde
f:\R^n\to \R$ is smooth then the validity of~\eqref{eq:def boundary}
is equivalent to the requirement that
%$\delta_\infty \widetilde f=0$ on $X\setminus Y$, where
\begin{equation}\label{eq"laplace}
\sum_{i=1}^n\sum_{j=1}^n\frac{\partial \widetilde f}{\partial x_i}
\cdot\frac{\partial \widetilde f}{\partial x_j}\cdot\frac{\partial^2
\widetilde f} {\partial x_i\partial x_j}=0\quad\mathrm{on}\
\R^n\setminus Y.
\end{equation}
If one interprets~\eqref{eq"laplace} in terms of viscosity
solutions, then it was proved in~\cite{Jen93} that the equivalence
of~\eqref{eq:def boundary} and~\eqref{eq"laplace} (when $Z=\R$ and
$X=\R^n$) holds for general Lipschitz $\widetilde f$. We refer to
the survey article~\cite{ACJ04} and the references therein for more
information on the many works that investigate this remarkable
connection between the classical Lipschitz extension problem and
PDEs.

Existence of isometric and isomorphic Lipschitz extensions has a
wide variety of applications in pure and applied mathematics.
Despite this rich theory, the issue raised by Aronsson's seminal
paper~\cite{Aro67} is that even when isometric Lipschitz extension
is possible, many such extensions usually exist, and it is therefore
natural to ask for extension theorems ensuring that the extended
function has additional desirable properties. In particular, the
notion of AMLE is an isometric Lipschitz extension which is locally
the ``best possible" extension. In this context, one can ask for
(appropriately defined) ``AMLE versions" of known Lipschitz
extension theories. As a first step, in light of
Theorem~\ref{thm:tree main} it is tempting to ask the following:
%(see the related results in \cite{SS10}):

\begin{question}\label{q:retract}
Let $Z$ be an absolute $1$-Lipschitz retract. Is it true that for
every length space $X$ and every closed subset $Y\subseteq X$, any
Lipschitz $f:Y\to Z$ admits an AMLE $\widetilde f :X\to Z$?
\end{question}
Note that unlike the situation when $Z$ is a metric tree, in the
setting of Question~\ref{q:retract} one cannot expect in general
that the AMLE will be unique: consider for example
$Z=\ell_\infty^2$, i.e., the absolute $1$-Lipschitz retract $\R^2$,
equipped with the $\ell_\infty$ norm. Let $X=\R$ and $Y=\{0,1\}$.
The $1$-Lipschitz mapping $f:Y\to \ell_\infty^2$ given by
$f(0)=(0,0)$, $f(1)=(1,0)$ has many AMLEs $\widetilde f:\R\to
\ell_\infty^2$, since for every $1$-Lipschitz function $g:\R\to \R$
with $g(0)=0$, $g(1)=1$, the mapping $x\mapsto (x,g(x))$ will be an
AMLE of $f$. At the same time, by using the existence of real-valued
AMLEs coordinate-wise, the answer to Question~\ref{q:retract} is
trivially positive when $Z=\ell_\infty(\Gamma)$ for any set
$\Gamma$.

While we do not give a general answer to Question~\ref{q:retract}, we show
here that general absolute $1$-Lipschitz retracts $Z$ do enjoy a
stronger Lipschitz extension property: $Z$-valued functions defined
on subsets of vertices of $1$-dimensional simplicial complexes
associated to unweighted finite graphs admit isometric Lipschitz
extensions which are $\infty$-harmonic. This issue, together with
the relevant definitions, is discussed in Section~\ref{sec:infty}
below. In addition to being crucially used in our proof of
Theorem~\ref{thm:tree main}, this result indicates that absolute
$1$-Lipschitz retracts do admit enhanced Lipschitz extension
theorems that go beyond the simple existence of isometric Lipschitz
extensions (which is the definition of absolute $1$-Lipschitz
retracts). At the same time, we describe below a simple example
indicating inherent difficulties in obtaining a positive answer to
Question~\ref{q:retract} beyond the class of metric trees (and their
$\ell_\infty$-products).

\subsection{$\infty$-harmonic functions and AMLEs on finite
graphs}\label{sec:infty}

Let $G=(V,E)$ be a finite connected (unweighted) graph. We shall
consider $G$ as a $1$-dimensional simplicial complex, i.e., the
edges of $G$ are present as intervals of length $1$ joining their
endpoints. This makes $G$ into a length space, where the
shortest-path metric is denoted by $d_G$. Given a vertex $v\in V$
denote its neighborhood in $G$ by $N_G(v)$, i.e., $N_G(v)=\{u\in V:\
uv\in E\}$.

%The neighborhood of a subset $W\subseteq V$, denoted
%$N_G(W)$, is defined as $N_G(W)=\bigcup_{w\in W} N_G(w)\setminus W$.

Let $(Z,d_Z)$ be a metric space.  We shall say that a function
$f:V\to Z$ is $\infty$-harmonic at $v\in V$ if there exist $u,w\in
N_G(v)$ such that \begin{equation}\label{eq:first infty}
d_Z(f(u),f(v))=d_Z(f(w),f(v))=\max_{z\in N_G(v)}
d_Z(f(z),f(v)),\end{equation} and
\begin{equation}\label{eq:second infty}
d_Z(f(u),f(w))=2\max_{z\in N_G(v)} d_Z(f(z),f(v)).
\end{equation}
$f:V\to Z$ is said to be $\infty$-harmonic on $W\subseteq V$ if it
is $\infty$-harmonic at every $v\in W$.

The connection to AMLEs is simple: for $\Omega\subseteq V$ and
$f:\Omega\to Z$, if $\widetilde f:G\to Z$ is an AMLE of $f$ then
$\widetilde f$ must be geodesic on edges, i.e., for $u,v\in V$ with
$uv\in E$, if $x\in G$ is a point on the edge $uv$ at distance
$\lambda\in [0,1]$ from $u$, then $d_Z(f(x),f(u))=\lambda
d_Z(f(u),f(v))$ and $d_Z(f(x),f(v))=(1-\lambda)d_Z(f(u),f(v))$
(apply~\eqref{eq:def boundary} to the open segment joining $u$ and
$v$). Moreover, if $G$ is triangle-free, then $\widetilde f$ is
$\infty$-harmonic on $V\setminus \Omega$. This follows from
considering in~\eqref{eq:def boundary} the open set $U\subseteq G$
consisting of the union of the half-open edges incident to $v\in
V\setminus \Omega$ (including $v$ itself). The vertices $u,w\in
N_G(v)=\partial U$ in~\eqref{eq:first infty} will be the points at
which $\Lip_{\partial U}\left(\widetilde f\right)$ is attained. The
restriction that $G$ is triangle-free implies that $d_G(u,w)=2$,
using which~\eqref{eq:second infty} follows from~\eqref{eq:def
boundary}.\footnote{In the above reasoning
the assumption that $G$ is triangle-free can be dropped if $Z$ is a
metric tree. But, this is not important for us:
 we only care about $G$ as a length space, and therefore we can replace each edge of $G$ by a path of length $2$,
resulting in a triangle-free graph whose associated $1$-dimensional
simplicial complex is the same as the original simplicial complex,
with distances scaled by a factor of $2$.}

The converse to the above discussion is true for mappings into
metric trees. This is contained in Theorem~\ref{thm:local-global}
below, whose simple proof appears in Section~\ref{s.infharmAMLE}. A
local-global statement analogous to Theorem~\ref{thm:local-global}
{\em fails} when the target (geodesic) metric space is not a metric
tree, as we explain in Remark~\ref{rem:flux} below.

Given a metric tree $T$, a finite graph $G=(V,E)$ and a function
$f:V\to T$, the {\em linear interpolation} of $f$ is the $T$-valued
function defined on the $1$-dimensional simplicial complex
associated to $G$ as follows: given an edge $e=uv\in E$ and $x\in e$
with $d_G(x,u)=\lambda d_G(u,v)$ and $d_G(x,v)=(1-\lambda)d_G(u,v)$,
the image $f(x)\in T$ is the point on the geodesic joining
$f(u)$ and $f(v)$ in $T$ with $d_T(f(x),f(u))=\lambda d_T(f(u),f(v))$ and
$d_T(f(x),f(v))=(1-\lambda)d_T(f(u),f(v))$.

\begin{theorem}\label{thm:local-global}
Let $T$ be a metric tree and $G=(V,E)$ a finite connected
(unweighted) graph. Assume that $\Omega\subseteq V$ and that $f:V\to
T$ is $\infty$-harmonic on $V\setminus \Omega$. Then the linear
interpolation of $f$ is an AMLE of $f|_{\Omega}$.
\end{theorem}

%$$
%d_T(x,y)+d_T(a,b)\le \max\left\{d_T(x,a)+d_T(y,b),d_T(x,b)+d_T(y,a)\right\}
%$$
%$$\Delta_\infty h \eqdef \frac{\sum_{i=1}^n\sum_{j=1}^n\frac{\partial h}{\partial x_i} \cdot\frac{\partial h}{\partial %x_j}\cdot\frac{\partial^2 h}{\partial x_i\partial x_j}}{\sum_{i=1}^n \left(\frac{\partial h}{\partial x_i}\right)^2}
%$$

\begin{remark}\label{rem:flux} {\em Consider the example depicted in Figure~\ref{fig:flux},
viewed as a $12$ vertex graph $G$ with vertices
$$
V=\{A,B,C,X,Y,Z\}\cup\{S_i\}_{i=1}^6 $$ and edges
$$E=\{XS_3,S_3A,AS_2,S_2B,AS_4,S_4C,BS_6,S_6C,ZS_5,S_5C,YS_1,S_1B\}.$$
(The role of the vertices $\{S_i\}_{i=1}^6$ is just to subdivide
edges so that the graph will be triangle-free.) The picture in
Figure~\ref{fig:flux} can also be viewed as a mapping $f:V\to \R^2$.
Denoting $\Omega=\{X,Y,Z\}$, this mapping is by construction
$\infty$-harmonic on $V\setminus \Omega$. In spite of this fact, the
linear interpolation of $f$ is not an AMLE of $f|_\Omega$. Indeed,
consider the open set $U=G\setminus \Omega$. Since the planar
Euclidean distance between any two of the points $f(X),f(Y),f(Z)$ is
strictly less than $3$ ($=$the distance between any two of the vertices
$\{X,Y,Z\}$ in $G$), we have $\Lip_{\partial
U}(f)=\Lip_{\{X,Y,Z\}}(f)<1$. At the same time, by considering the
vertices $A,B,C$ we see that $\Lip_U(f)=1$.}
\begin{figure}[h]\label{fig:flux}
\begin{center}\includegraphics[scale=0.417]{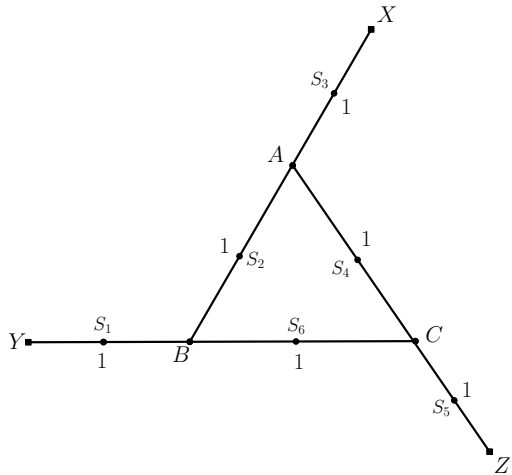}
\end{center}
 \caption{{{ An example of an $\infty$-harmonic function which isn't an AMLE.}}} \label{fig:hull}
\end{figure}
\end{remark}

In Section~\ref{s.infharmAMLE} we show that absolute $1$-Lipschitz
retracts have a stronger Lipschitz extension property, namely they
admit $\infty$-harmonic extensions for functions from finite graphs:

\begin{theorem}\label{thm:aronszajn}
Assume that $(Z,d_Z)$ is an absolute $1$-Lipschitz retract and that
$G=(V,E)$ is a finite connected (unweighted) graph. Fix
$\Omega\subseteq V$ and $f:\Omega\to Z$. Then there exists a mapping
$\widetilde f:V\to Z$ which is $\infty$-harmonic on $V\setminus
\Omega$ such that
$$
\left.\widetilde f\right|_\Omega=f\quad\mathrm{and}\quad
\Lip_V\left(\widetilde f\right)=\Lip_\Omega(f).
$$
\end{theorem}

The existence part of Theorem~\ref{thm:tree main} is deduced in
Section~\ref{s.existence} from Theorem~\ref{thm:aronszajn} via a
 compactness argument that relies on a
comparison-based characterization of AMLE that we establish in
Section \ref{s.treecomparison}. The uniqueness part of
Theorem~\ref{thm:tree main} is proved via a topological argument
(and the results of Section~\ref{s.treecomparison}) in
Section~\ref{s.uniqueness}.

\subsection{Tug of War and Politics}

In the special case when $T\subseteq \R$ is an interval,
Theorem~\ref{thm:tree main} was proved in \cite{PSSW09} without the
local compactness assumption using a two-player, zero-sum stochastic
game called Tug of War.  We expect that one could adapt the
arguments in \cite{PSSW09} and the game called Politics (introduced
below) to give a proof of Theorem~\ref{thm:tree main} that does not
use local compactness; however, this would involve rewriting large
sections of \cite{PSSW09} in a significantly more complicated way,
and we will not attempt to do this here.

 {\bf Tug of War} is a two-player, zero-sum
stochastic game. In this game, one starts with an initial point $x_0
\in X \setminus Y$; then at the $k$th stage of the game, a fair coin
is tossed and the winner gets to choose any $x_k \in X$ with $|x_k -
x_{k-1}| < \epsilon$.  Informally, the winning player ``tugs'' the
game position up to $\epsilon$ units in a direction of her choice.
The game ends the first time $K$ that $x_K \in Y$, and player one
collects a payoff of $f(x_K)$ from player two.  It was shown that as
$\epsilon \to 0$, the value of the game (informally, the amount the
first player wins in expectation when both players play optimally;
see Section \ref{politicalsection}) tends to $\widetilde f(x_0)$.
In addition to its usefulness in proofs, the game theory provides a
deeper understanding of what an AMLE {\em is}.  Although AMLEs are
often difficult to compute explicitly, one can always provide upper
and lower bounds by giving explicit strategies for the game and
showing that they guarantee a certain expected payoff for one player
or the other.  It is therefore natural to ask for an analog of Tug
of War that makes sense when $T$ is not an interval.

Since $\widetilde f(x_0)$ is a point in $T$, however, and not in
$\R$, it is not immediately obvious how $\widetilde f(x_0)$ can
represent a value for either player.  We will solve this problem by
augmenting the state space of the game to include declared
``targets'' $t_k, o_k \in T$ as well as ``game positions'' $x_k \in
X$.   Before explaining this, we remark that one obtains a slight
generalization of Tug of War by letting $x_k$ be vertices of any
(possibly infinite) graph with vertex set $X$ and $Y \subseteq X$.
One then requires that $x_k$ and $x_{k-1}$ be adjacent in that graph
(instead of requiring $|x_k-x_{k-1}| < \epsilon$).  We now introduce
the game of {\bf Politics} in a similar setting.

Let $G=(V,E)$ be an unweighted undirected graph which may have self
loops. Fix $Y\subseteq  V$ and a mapping $f:Y\to T$. Begin with an
initial game position $x_0 \in V \setminus Y$ and an initial ``target''
$t_0 \in T$.  At the $k$th round of the game, the players determine
the values $(x_k, t_k)$ as follows:
\begin{enumerate} \item Player \I chooses an ``opposition target'' $o_k \in T$ and collects
$d_T(o_k,t_{k-1})$ units from player \II.
\item Player \II chooses a new target $t_k \in T$ and collects $d_T(o_k,t_k)$ units from player \I.
\item A fair coin is tossed and the winner of the toss chooses a new game
position $x_k\in X$ with $\{x_{k-1},x_k\} \in E$.
\end{enumerate}
The total amount player \I gains at each round is $d_T(o_k,t_{k-1})
- d_T(o_k, t_k)$. Similarly, player \II gains $d_T(o_k,t_k) -
d_T(o_k, t_{k-1})$ at each round. The game ends after round $K$,
where $K$ is the smallest value of $k$ for which $x_k \in Y$.  At
this point player \I collects an additional $d_T(f(x_k), t_k)$ units
from player \II.  (If the game never ends, we declare the total
payout for each player to be zero.)

The game is called ``Politics'' because we may view it as a model
for a rather cynical zero-sum political struggle in which $f(x_K)$
represents a ``political outcome,'' but both parties care only about
their own perceived political strength, and not about the actual
outcome.  We think of the target as representing the ``declared
political objective'' of player \II; the terminal payoff rule, makes
it clear that player \II would prefer $f(x_K)$ be close to this
declared target (in order to ``appear successful'').  Player \II is
allowed to adjust the target during each round, but loses points for
moving her target closer to the declared opposition target $o_k$
(because ``making a concession'' makes her appear weak) and gains
points for moving her target further from the opposition target
because ``taking a harder line'' makes her appear strong).
\footnote{There is a more player-symmetric variant of this game in
which each player, upon moving a target, earns the net change in the
distance from the opponent's target. That is, player \II earns
$d_T(o_k, t_k) - d_T(o_k,t_{k-1})$ when choosing $t_k$ (so player \I
earns $d_T(o_k,t_{k-1})-d_T(o_k, t_k)$) and player \I earns
$d_T(o_k,t_{k-1}) - d_T(o_{k-1}, t_{k-1})$ when choosing $o_k$. In
fact, by combining like terms, modifying the end-of-game payout
function, and defining $o_0 = t_0$, one can make this game {\em
equivalent} to the one described above but with twice the total
payout.}

We will prove the following for finite graphs:
\begin{proposition} \label{towvalue}
Fix a finite graph $G=(V, E)$, some $Y\subseteq V$, a
 metric tree $T$, and a function $f:Y \to T$.  View
$G$ as a length space (with all edges having length one) and let
$\widetilde f:G \to T$ be the AMLE of $f$.  Then the value of the
game of Politics with these parameters and initial vertex $x_0\in
V\setminus Y$ is given by
$$d_T\left(\widetilde f(x_0), t_0\right).$$
\end{proposition}

Proposition \ref{towvalue} will be proved in Section \ref{politicalsection}.  An extension of Proposition \ref{towvalue} to infinite graphs (via the methods of \cite{PSSW09}) is probably possible, but we will not attempt it here.

%A local characterization of the AMLE as an {\em infinity harmonic} function (appropriately defined for tree-valued %maps) will be given in Section \ref{s.infharmAMLE}.

\subsection{Some open questions and directions for future research}\label{sec:future} It would be of interest to understand known isometric extension theorems in the context of the AMLE problem. Specifically, we ask:
\begin{question}\label{q:kirszbraun}
Is there an AMLE version of Kirszbraun's extension theorem, i.e, is it true that for
every pair of Hilbert spaces $H_1,H_2$ and every closed subset $Y\subseteq H_1$, any
Lipschitz mapping $f:Y\to H_2$ admits an AMLE $\widetilde f :H_1\to H_2$?
\end{question}
We refer to the manuscript~\cite{SS10} for a discussion of subtleties related to Question~\ref{q:kirszbraun}, as well as some partial results in this direction. Examples of additional isometric extension theorems that might have AMLE versions are contained in~\cite{Val44,Val45,MR0461107,LS97,Nao01}.

 The study of isomorphic extensions in the context of the AMLE problem is wide open. Since when Lipschitz extension is possible a constant factor loss is usually necessary, and since isomorphic extensions suffice for many applications, it would be of interest if some isomorphic extension theorems had ``almost locally optimal" counterparts. For example, one might ask for the existence of a constant $K>0$ such that one can extend any mapping $f:Y\to Z$ to a mapping $\widetilde f : X\to Z$ so that for every open $U\subseteq X\setminus Y$ we have
 \begin{equation}\label{eq:isomorphic formulation}
 \Lip_U\left(\widetilde f\right)\le K\cdot \Lip_{\partial U}\left(\widetilde f\right).
 \end{equation}
 %Note that one can't expect uniqueness statements in the context of~\eqref{eq:isomorphic formulation}.
 Examples of isomorphic extension results that could be studied in the context of the AMLE problem include~\cite{Lin64,MP84,JL84,JLS86,Ball92,PY95,LPS00,BS02,LN05,NPSS06,MR2340707,MN06,Kal07,MR2200122}. Unlike isometric extension theorems, isomorphic extension theorems cannot be done ``one point at time", since na\"\i vely the constant factor losses at each step would accumulate. For this reason, isomorphic extension theorems usually require methods that are very different from their isometric counterparts. One would therefore expect that entirely new approaches are necessary in order to prove AMLE versions of isomorphic extension.

%\begin{comment}

\subsection{Possible applications} The image processing literature makes use of real-valued AMLEs as a technique for image inpainting and surface reconstruction --- see~\cite{CMS98,ACGR02,hm-lips,CHSV08}. Since many data sets in areas ranging from computer science to biology have a natural tree structure, it stands to reason that problems involving reconstruction/interpolation of missing tree-valued data could be similarly approached using tree valued AMLEs.
%Thus, we believe that the notion of AMLE of tree-valued mappings is also natural from an applied point of view.

Tree-valued AMLEs may also be useful for problems that do not
involve trees {\em a priori}. To give a simple illustration of this,
suppose we have a two-dimensional surface $S$ embedded in $\R^3$
that separates an ``inside" from an ``outside," but such that on
some open $W \subseteq \R^3$ the shape of the surface is not known.
Let $d(x)$ be the signed distance of $x$ from $S$ (i.e., the actual
distance if $x$ is on the outside and minus that distance if $x$ is
on the inside).  If we can compute or approximate $d(x)$ outside of
of $W$, then the extension of $d(x)$ to $W$ has a zero set that can
be interpreted as a ``reconstructed'' approximation to $S$.  This
approach and related methods are explored in \cite{CHSV08}.

If instead of a single ``inside'' and ``outside'' there were three
or more regions of space meeting at a point $v$, and the union $S$
of the interfaces between these regions was unknown in a
neighborhood $W$ of $v$, then we could use the same approach but
replace $\R$ with the metric tree $\bigcup \omega_i [0,\infty)
\subseteq \C$ for some complex roots of unity $\omega_i$, and let
$d(x)$ be $\omega_i$ (when $x$ is in the $i$th region) times the
distance from $x$ to $S$.  A similar technique could be used for
inpainting a two-dimensional image comprised of a small number of
monochromatic regions. Indeed, for such problems, it is not clear
how one could apply the AMLE method without using trees.

%(The example above used star trees, but one can imagine more general
%contexts where one might have partition of a full measure subset of
%$\R^d$ into open sets $A_1, \ldots, A_k$ and a Lipschitz map from
%$\R^d$ to $\R$ that sent each $A_i$ to a distinct edge of a tree; if
%one knew such a map only on part of the space, the AMLE extension
%would provide one way to extend the map and thereby reconstruct the
%$A_i$.)

%These simple examples (with, in particular, very simple tree structures) indicate that much like the use of real-valued AMLE as a method to interpolate missing numerical data,

%\end{comment}

\section{Comparison formulation of absolute minimality} \label{s.treecomparison}

We take the following definition from \cite{MR2341302} (see
also~\cite{Jen93,CEG01} for the case $X=\R^n$). Let $U$ be an open
subset of a length-space $(X,d_X)$ and let $f: \overline U\to\R$ be
continuous. Then $f$ is said to satisfy \textbf{comparison with
distance functions from above} on $U$ if for every open $W\subseteq
U$,  $z\in X\setminus W$, $b\ge 0$ and $c\in \R$ we have the
following:
\begin{equation} \label{e.distancecomparison}\Big( \forall\ x\in \partial W\ \ f(x)\le
b\,d_X(x,z)+c\Big) \implies \Big(\forall\ x\in W\ \  f(x)\le
b\,d_X(x,z)+c\Big).\end{equation} The function $f$ is said to
satisfy \textbf{comparison with distance functions from below} on
$U$ if the function $-f$ satisfies comparison with distance
functions from above on $U$, i.e., for every open $W\subseteq U$,
$z\in X\setminus W$, $b\ge 0$ and $c\in \R$ we have the following:
\begin{equation} \label{e.distancecomparison-below}\Big( \forall\ x\in \partial W\ \ f(x)\ge
-b\,d_X(x,z)+c\Big) \implies \Big(\forall\ x\in W\ \  f(x)\ge
-b\,d_X(x,z)+c\Big).\end{equation} Finally, $f$ satisfies
\textbf{comparison with distance functions} on $U$ if it satisfies
comparison with distance functions from above and from below on $U$.
We cite the following:
\begin{proposition}[\cite{MR2341302}] \label{AMiffCDF}
Let $U$ be an open subset of a length space.
  A continuous $f:\overline U\to \R$ satisfies comparison with distance functions on
  $U$ if and only if it is an AMLE of $f|_{\partial U}$.
\end{proposition}

\begin{remark} \label{r.boundaryw} {\em The definition of comparison with distance functions from above
would not change if we added the requirement that $z \not \in
\partial W$; if \eqref{e.distancecomparison} or~\eqref{e.distancecomparison-below} fails and $z \in \partial W$,
then it will fail (with a modified $c$) when $W$ is modified to
include some neighborhood of $z$. The definition would also not
change if we required $b > 0$.  If \eqref{e.distancecomparison}
or~\eqref{e.distancecomparison-below} fails with $b=0$, then it
fails for some sufficiently small $b'>0$.}
\end{remark}

We will need to have an analog of the above definition with the real
line $\R$ replaced with $T$.  The definition makes sense when $T$ is
any metric space, but we will only use it in the case when $T$ is a
metric tree.  We say $f:\overline U \to T$ satisfies
\textbf{$T$-comparison} on $U$ if for every $t \in T$, the function
$x\mapsto d_T(t, f(x))$ satisfies comparison with distance functions
from above on $U$.  This generalizes comparison with distance
functions:

\begin{proposition}\label{prop:same for line} If $T$ is the closed interval $[t_1, t_2] \subseteq \R$, then $f: \overline U \to T$
satisfies $T$-comparison on $U$ if and only if it satisfies comparison with distance functions on $U$.
\end{proposition}
\begin{proof}
If $f$ satisfies $T$-comparison on $U$, then the mappings $x\mapsto
d_T(t_1, f(x))=f(x)-t_1$ and $x\mapsto d_T(t_2,f(x))=t_2-f(x)$
satisfy comparison with distance functions from above on $U$, hence
$f$ and $-f$ both satisfy comparison with distance functions from
above. Conversely, if $f$ satisfies comparison with distance
functions on $U$, then for all $t\in [t_1,t_2]$ the mapping
$x\mapsto d_T(f(x), t) = (f(x)-t) \vee (t-f(x))$ satisfies
comparison with distance functions from above because it is a
maximum of two functions with this property.
\end{proof}

Proposition \ref{AMiffCDF} also has a natural generalization, which
is contained in Proposition~\ref{p.treecomparison} below. Note that
the proof of this generalization uses the assumption that $T$ is a
metric tree in the ``only if'' direction; for the ``if'' direction $T$
can be any metric space.

\begin{proposition} \label{p.treecomparison}
Let $U$ be an open subset of a length space $(X,d_X)$, and let
$(T,d_T)$ be a metric tree. A continuous function $f:\overline U\to
T$ satisfies $T$-comparison on $U$ if and only if it is an AMLE of
$f|_{\partial U}$.
\end{proposition}

\begin{proof}
We will first suppose, to obtain a contradiction, that $f$ is not an
AMLE of $f|_{\partial U}$, but satisfies $T$-comparison. Then there
is an open $W \subseteq U$ such that $\Lip_W(f) > \Lip_{\partial
W}(f)$. That is, there is a path $P$ in $X$ connecting points $x$
and $y$ in $W$ whose length $L$ satisfies
\begin{equation}\label{eq:length for cotradiction}
\frac{d_T(f(x),f(y))}{L}
>  \Lip_{\partial W}(f).
\end{equation}
If $y_1$ and $y_2$ are the first and last times $P$ hits $\partial
W$,  $d_T(f(y_1),f(y_2)) \leq \Lip_{\partial W}(f)\cdot d_X(y_1,
y_2)$; hence the property~\eqref{eq:length for cotradiction} holds
for either the portion of $P$ between $x$ and $y_1$ or the portion
between $y_2$ and $y$. Thus, we may take $P$ to be entirely
contained in $W$; replacing $P$ with a slightly shorter sub-path of
$P$, we may assume the endpoints of $P$ are both in $W$ as well, and
that $P$ is some positive distance $\delta$ from $\partial W$.  Set
$$m\eqdef\frac{d_T(f(x),f(y))}{L}>  \Lip_{\partial W}(f).$$  We may then find $x_1$ arbitrarily
close to some fixed point $x_0$ along $P$ satisfying
$$\frac{d_T(f(x_0),f(x_1))}{d_X(x_0,x_1)} \geq m >  \Lip_{\partial W}(f).$$

Now we consider the distance function $m d_X(x_0,\cdot)$. We will
compare it to the function $d_T(f(x_0),f(\cdot))$.  Since the latter
is at least as large as the former at the point $x_1$,
$T$-comparison implies that it must be at least as large at some
point on $\partial W$.  This implies that for any $\epsilon > 0$ we
may find a $z \in \partial W$ where $$m'\eqdef
\frac{d_T(f(x_0),f(z))}{d_X(x_0,z)} \geq m-\epsilon.$$  In
particular, we may assume $m' > \Lip_{\partial W}(f)$.  Next choose
$m'' \in \big(\Lip_{\partial W}(f), m'\big)$. Consider the distance
function $m'' d_X(z, \cdot)$ and compare it to $d_T(f(z),
f(\cdot))$. Since the functions are equal at $z$ and the latter is
larger than the former at $x_0$, the latter must be larger than the
former at some point $w \in (\partial W)\setminus \{z\}$. But this
implies
$$\frac{d_T(f(z), f(w))}{d_X(z,w)} > \Lip_{\partial W}(f),$$ a
contradiction.

We now proceed to the converse. Note that since $T$ is a bounded
metric space, by intersecting $U$ with a large ball it suffices to
prove the converse when $U$ is bounded. Suppose, to obtain a
contradiction, that $f$ is an AMLE of $f|_{\partial U}$ and does not
satisfy $T$-comparison on $U$. Since $f$ does not satisfy
$T$-comparison on $U$, there exists an open $W \subseteq U$, a point
$x_0 \notin W$ and $c\in\R$, $b\ge 0$, such that for some $t\in T$
we have $d_T(t,f(x))\le bd_X(x_0,x)+c$ for all $x\in
\partial W$, yet $d_T(t,f(y))>bd_X(x_0,y)+c$ for some $y\in W$.
Write $F(z)=bd_X(x_0,z)+c$. We may replace $W$ with the connected
component of $\{x\in W :\  d_T(t,f(x)) > F(x) \}$ containing $y$, so
that one has $d_T(t,f(x)) = F(x)$ at the boundary of $W$.  By
looking at a nearly-shortest path from $y$ to $x_0$, we deduce that
$\Lip_W(f)
> b$. If we could also show that $\Lip_{\partial W}(f) = b$ (which is
trivially the case when $T \subseteq \mathbb R$, but not for a more
general metric tree $T$) we would have a contradiction to the AMLE
property of $f$. Instead of proving this for the particular $W$
constructed above, we will show that there exists a smaller $W$ for
which the analogous statement holds.

Consider the  function $$G(s) \eqdef \sup_{\substack{x \in \overline{W}\\
d_X(x_0,x) = s}} d_T(t,f(x)),$$ which is defined on the interval
$[s_1, s_2]$, where $s_1$ and $s_2$ are the infimum and supremum of
the set $\{d_X(x_0,x):\  x \in W\}$, respectively.
\begin{comment}
Note that $G$ is upper semi-continuous, i.e., if
$\{\sigma_n\}_{n=1}^\infty\subseteq [s_1,s_2]$ converges to
$\sigma_\infty$ then $\limsup_{n\to\infty} G(\sigma_n)\le
G(\sigma_\infty)$. Indeed, by local compactness there exists $z_n\in
\overline W$ such that $d_X(z_n,x_0)=\sigma_n$ and
$G(\sigma_n)=d_T(t,f(z_n))$. Denote $\ell= \limsup_{n\to\infty}
G(\sigma_n)$. By passing to a subsequence (using local compactness
once more), assume that $\ell=\lim_{n\to\infty} G(\sigma_n)$, and
that $\lim_{n\to\infty} z_n=z_\infty\in \overline W$ exists. Then
$d_X(z_\infty,x_0)=\sigma_\infty$, and therefore
$G(\sigma_\infty)\ge d_T(t,f(z_\infty))=\lim_{n\to\infty}
d_T(t,f(z_\infty))=\ell$, as required.
\end{comment}
By assumption $G(s)$ lies above the line $ bs + c$ for some $s\in
[s_1,s_2]$, though not for $s_1$ and $s_2$. Hence, if we define
\begin{equation*}\label{eq:def:M}
M\eqdef \sup \big\{G(s)-bs-c:\ s\in [s_1,s_2]\big\}
\end{equation*}
then $M>0$. Write
\begin{equation*}\label{eq:def S}
S\eqdef \left\{\sigma\in [s_1,s_2]:\
\limsup_{\substack{s\to\sigma\\s\in[s_1,s_2]}}\big(G(s)-bs-c\big)=M\right\},
\end{equation*}
and note that $S$ is a nonempty closed subset of $[s_1,s_2]$, so that
$s_0\eqdef \inf S\in S$.

\begin{comment}
The upper semi-continuity of $G$ implies that the function $s\mapsto
G(s)-bs-c$ attains its (positive) maximum $M$ on $[s_1,s_2]$, and
that moreover there exists $s_0\in (s_1,s_2)$ such that
$G(s_0)-bs_0-c=M$ and that if $G(s)-bs-c=M$ then $s\ge s_0$.
\end{comment}

For $\e>0$ and $x\in X$ define
$$
F_\e(x)\eqdef (b+\e)d_X(x_0,x)+M+c-\e s_0-\e^2,
$$
and
$$
W_\e\eqdef\big\{x\in W:\ d_T(t,f(x))>F_\e(x)\big\}.
$$
Observe that $W_\e\neq\emptyset$ for all $\e>0$. To see this fix
$\delta>0$. Since $s_0\in S$ there exists $s\in [s_1,s_2]$ such that
$|s-s_0|\le \delta $ and $G(s)-bs-c\ge M-\delta$. By the definition
of $G(s)$, there is $z_0\in \overline W$ satisfying $d_X(z_0,x_0)=s$
and $G(s)\le d_T(t,f(z_0))+\delta$. Since $f$ is continuous at
$z_0$, there is $\eta\in (0,\delta)$ such that if $d_X(z,z_0)<\eta$
then $d_T(f(z),f(z_0))<\delta$. Take $z\in W$ with
$d_X(z,z_0)<\eta$. Then,
\begin{eqnarray*}
d_T(t,f(z))&>&d_T(t,f(z_0))-\delta\\&\ge&G(s)-2\delta\\&\ge&M+bs+c-3\delta\\&\ge
& M+bs_0+c-(3+b)\delta\\&=&F_\e(z)-(b+\e)d_X(x_0,z)+bs_0+\e
s_0+\e^2-(3+b)\delta\\&>&F_\e(z)-(b+\e)(s+\eta)+bs_0+\e
s_0+\e^2-(3+b)\delta\\&\ge&F_\e(z)-(b+\e)(s_0+\delta+\eta)+bs_0+\e
s_0+\e^2-(3+b)\delta\\&>&F_\e(z)+\e^2-(3\delta+2\e\delta+3b\delta).
\end{eqnarray*}
Thus for $\delta$ small enough we have $z\in W_\e$. The following
claim contains additional properties of the sets $W_\e$ that we will
use later.

\begin{claim}\label{claim}
The open sets $\{W_\e\}_{\e>0}$ have the following properties:
\begin{enumerate}
\item If $0<\e_1<\e_2$ then $\overline{W_{\e_1}}\subseteq W_{\e_2},$
\item $\lim_{\e\to 0} \sup_{x\in W_\e}\big| d_X(x_0,x)-s_0\big|=0,$
\item $\lim_{\e\to 0} \sup_{x\in W_\e}\big| d_T(t,f(x))-(M+bs_0+c)\big|=0.$
\end{enumerate}
\end{claim}
\begin{proof}
Fix $0<\e_1<\e_2$ and $x\in \overline{W_{\e_1}}$. Write
$s=d_X(x_0,x)$. Since $d_T(t,f(x))\ge F_{\e_1}(x)$, we have
$G(s)-bs-c\ge\e_1 s+M-\e_1 s_0-\e_1^2$. By the definition of $M$,
this implies that $s\le s_0+\e_1$. Hence,
\begin{multline*}
d_T(t,f(x))\ge(b+\e_1)s+M+c-\e_1
s_0-\e_1^2=F_{\e_2}(x)+(\e_2-\e_1)s_0+\e_2^2-\e_1^2-(\e_2-\e_1)s\\
>F_{\e_2}(x)+(\e_2-\e_1)s_0+\e_2^2-\e_1^2-(\e_2-\e_1)(s_0+\e_1)=F_{\e_2}(x)+\e_2(\e_2-\e_1)>F_{\e_2}(x).
\end{multline*}
Thus $x\in W_{\e_2}$, proving the first assertion of
Claim~\ref{claim}.

To prove the second assertion of Claim~\ref{claim}, note that we
have already proved above that if $x\in W_\e$ then $d_X(x_0,x)\le
s_0+\e$. Thus, if the second assertion of Claim~\ref{claim} fails
there is some $\delta>0$ and a sequence
$\{\e_n\}_{n=1}^\infty\subseteq [0,1]$ with
$\lim_{n\to\infty}\e_n=0$, such that for each $n\in \N$ there is
$z_n\in W_{\e_n}$ with $d_X(z_n,x_0)\le s_0-\delta$. Write
$\sigma_n=d_X(z_n,x_0)$, and by passing to a subsequence assume that
$\lim_{n\to \infty} \sigma_n=\sigma_\infty$ exists. Then
$\sigma_\infty\le s_0-\delta$ and,
\begin{multline*}\limsup_{n\to\infty} \big(G(\sigma_n)-b\sigma_n-c\big)\ge \limsup_{n\to\infty} d_T(t,f(z_n))-b\sigma_\infty-c\ge \limsup_{n\to\infty} F_{\e_n}(z_n)-b\sigma_\infty-c\\=
\limsup_{n\to\infty}\big((b+\e_n)\sigma_n+M+c-\e_ns_0-\e_n^2\big)-b\sigma_\infty-c=M.\end{multline*}
Thus $\sigma_\infty\in S$. But since $\sigma_\infty\le
s_0-\delta$, this contradicts the choice of $s_0$ as the minimum of $S$. The proof of the second
assertion of Claim~\ref{claim} is complete. The third assertion of
Claim~\ref{claim} now follows, since if $x\in W_\e$ then by writing
$s=d_X(x,x_0)$ we see that
\begin{multline*}
b(s-s_0)\ge [(G(s)-bs-c) -M] + b(s-s_0)\ge
d_T(t,f(x))-(M+bs_0+c)\\\ge F_\e(x)-(M+bs_0+c)=b(s-s_0)+\e s-\e
s_0-\e^2.
\end{multline*}
Thus
$$
\sup_{x\in W_\e}\big| d_T(t,f(x))-(M+bs_0+c)\big|\le b\sup_{x\in
W_\e}\big| d_X(x_0,x)-s_0\big|+\e s_0+\e^2,
$$
and therefore the third assertion of Claim~\ref{claim} follows from
the second assertion of Claim~\ref{claim}.
\end{proof}

We are now in position to conclude the proof of
Proposition~\ref{p.treecomparison}. Let $V$ be the the set of
vertices of the metric tree $T$. We claim that for all $\e>0$ such
that $\e s_0+\e^2\le M$ we have $f(W_\e)\cap V\neq \emptyset$
(recall that by our assumption we have $M>0$). Indeed, if $x\in
\partial W_\e$ then either $d_T(t,f(x))=F_\e(x)$ or $x\in \partial
W$. In the latter case, by assumption we have $d_T(t,f(x))\le
bd_X(x_0,x)+c\le F_\e(x)$, where the last inequality follows from
$\e s_0+\e^2\le M$. Thus, by the definition of $W_\e$, the function
$f$ does not satisfy $T$-comparison on $W_\e$. Since $f$ is an AMLE of $f|_{\partial U}$, Proposition
\ref{AMiffCDF}, combined with Proposition~\ref{prop:same for line},
now implies that $f|_{W_\e}$ must take values in $V$.

Due to part $(1)$ of Claim~\ref{claim}, there exists $v\in V$ such
that $v\in \bigcap_{\e>0} W_\e$. Let $W_\e'$ be the connected
component of $W_\e$ whose image under $f$ contains $v$. By part
$(3)$ of Claim~\ref{claim}, for $\e$ small enough we have
$f(W_\e')\cap V=\{v\}$. Since, by the definition of $W_\e$ and the
connectedness of $W_\e'$, for $x\in \partial W_\e'$ we have
$d_T(t,f(x))=F_\e(x)$, by considering a nearly-shortest path from a
point in $W_\e'$ to $x_0$ we see that $\Lip_{W_\e'}(f)>b+\e$. Since
$f$ is an AMLE of $f|_{\partial U}$, it follows that $\Lip_{\partial
W_\e'}(f)>b+\e$. This implies that there are distinct $x_\e,y_\e\in
\partial W_\e'$ such that
$d_T(f(x_\e),f(y_\e))>(b+\e)d_X(x_\e,y_\e)$. But, since
$d_T(t,f(x_\e))=F_\e(x_\e)$ and $d_T(t,f(y_\e))=F_\e(y_\e)$, it must
be the case that the distance from $t$ of both $f(x_\e)$ and
$f(y_\e)$ is at least their distance from $v$. Indeed, if $t\in
\bigcap_{\e>0} f(W_\e')$ then it would follow from part $(3)$ of
Claim~\ref{claim} that $t=v$, and there is nothing to prove.
Otherwise, for $\e$ small enough $t\notin f(W_\e')$, and therefore,
since $T$ is a tree, if at least one of the points $f(x_\e),f(y_\e)$ is closer to $t$ than to $v$ then the points $t,f(x_\e),f(y_\e)$ all lie on
the same geodesic in $T$, implying that:
\begin{multline*}
d_T(f(x_\e),f(y_\e))=\big|d_T(t,f(x_\e))-d_T(t,f(y_\e))\big|=|F_\e(x_\e)-F_\e(y_\e)|\\=
(b+\e)\big|d_X(x_0,x_\e)-d_X(x_0,y_\e)\big|\le (b+\e)d_X(x_\e,y_\e),
\end{multline*}
a contradiction to the choice of $x_\e,y_\e$.

Having proved that  both $f(x_\epsilon)$ and $f(y_\epsilon)$ lie
further away from $t$ than $v$, if we consider a nearly
shortest-path between $x_\e$ and $y_\e$, it must include points
$x_1,x_2\in \overline{W_\e'}$ such that $f(x_1)$ and $f(x_2)$ lie on the same component
$I$ of $T\setminus V$ --- on the other side of $v$ from $t$ --- and
satisfy
\begin{equation}\label{eq:violation}
\frac{d_T(f(x_1),f(x_2))}{d_X(x_1,x_2)}>(b+\e).
\end{equation}
Suppose that $f(x_1)$ is closer to  $t$ that $f(x_2)$. Note that
\begin{multline*}
d_T(f(x_1),f(x_\e))=d_T(t,f(x_\e))-d_T(t,f(x_1))\le
F_\e(x_\e)-F_\e(x_1)\\=(b+\e)\big(d_X(x_0,x_\e)-d_X(x_0,x_1)\big)\le
(b+\e)d_X(x_\e,x_1).
\end{multline*}
Moreover, if $f(x)=x_1$ then trivially $d_T(f(x_1),f(x))\le
(b+\e)d_X(x_1,x)$. Thus, if we let $J\subseteq I$ be the open
interval joining $f(x_\e)$ and $f(x_1)$, then
$d_T(f(x_1),f(\cdot))\le (b+\e)d_X(x_1,\cdot)$ on $\partial
\left(f^{-1}(J)\cap W_\e'\right)$. By~\eqref{eq:violation} we now
have a violation of $T$-comparison on $f^{-1}(J)$, which contradicts
Proposition~\ref{AMiffCDF}.
\end{proof}

\section{Uniqueness} \label{s.uniqueness}

In this section we prove the uniqueness half of Theorem \ref{thm:tree main} (which does not require the locally compact assumption) as Lemma \ref{l.uniquenesshalf} below.  Before doing so, we prove some preliminary lemmas.

\begin{lemma} \label{l.compabovelocal} Suppose that $X$ is a length space, that $Y \subseteq X$ is closed, that $f:Y \to \mathbb R$ is Lipschitz and bounded, and that $\widetilde f:X \to \R$ is the AMLE of $f$.  (Existence and uniqueness of $\widetilde f$ are proved in \cite{PSSW09}.)  Suppose that $g:X\to \R$ is another bounded and continuous extension of $f$, and that for some fixed $\delta > 0$, this $g$ satisfies comparison with distance functions from above on every radius $\delta$ ball centered in $X \setminus Y$.  Then $g \leq \widetilde f$ on $X$.
\end{lemma}

\begin{proof}
This is proved (though not explicitly stated) in \cite{PSSW09}.  Precisely, it is shown there that for $\e>0$, the comparison with distance functions from above on balls of radius larger than $2 \epsilon$ implies that the first player in a modified tug of war game (with game position $v_k$ and step size $\epsilon$) can make $g(v_k)$ a submartingale until the termination of the game, which in turn implies that $g \leq f_\epsilon$ where $f_\epsilon$ is the value of this game.  It is also shown that $\lim_{\epsilon \to 0} f_\epsilon = \widetilde f$ holds on $X$.  Taking $\epsilon \to 0$ (and noting $2 \epsilon < \delta$ for small enough $\epsilon$) gives $g \leq \widetilde f$.
\end{proof}

The following was proved in \cite[Lem.~ 5]{AS10}.  The statement in \cite{AS10} was only made for the special case $X \subseteq \R^n$, but the (short) proof was not specific to $\mathbb R^n$.  For completeness, we copy the proof from \cite{AS10}, adapted to our notation.  We will vary
the presentation just slightly --- using suprema over open balls instead of maxima over closed balls --- because in our context (since we do not assume any kind of local compactness) maxima of continuous functions on closed balls are not necessarily obtained.

\begin{lemma} \label{l.armstrongsmart} Let $(X,d_X)$ be a length space, $x_0\in X$ and $\e>0$.
Suppose that $f:X\to \R$ satisfies comparison with distance functions from above on a domain containing $B(x_0,2\e)$.  Write $$f^\epsilon(x) \eqdef \sup_{B(x,\e)} f,\quad f_\epsilon(x) \eqdef \inf_{B(x,\e)} f$$and $$S^+_\epsilon f(x) \eqdef \sup_{y \in B(x,\e)} \frac{f(y) - f(x)}{\epsilon}, \quad S^-_\epsilon f(x) \eqdef \sup_{y \in B(x,\e)} \frac{f(x) - f(y)}{\epsilon}.$$
Then $$S_\epsilon^- f^\epsilon(x_0) \le S^+_\epsilon f^\epsilon(x_0) .$$ \end{lemma}

\begin{proof}
For $\delta > 0$ we may select $y_0 \in B(x_0,\e)$ and $z_0 \in B(x_0,2\epsilon)$ such that $|f(y_0) - f^\epsilon(x_0)|\le \delta$ and $|f(z_0) - f^{2 \epsilon}(x_0)|\le\delta$.  Then,
\begin{eqnarray}\label{eq:use def}
\nonumber\epsilon\left(S^-_\epsilon f^\epsilon(x_0) - S^+_\epsilon f^\epsilon(x_0)\right) &=& 2 f^\epsilon(x_0) - (f^\epsilon)^\epsilon(x_0) - (f^\epsilon)_\epsilon(x_0)  \\ \nonumber  &\leq&  2 f^\epsilon(x_0) - f^{2\epsilon}(x_0) - f(x_0) \\& \leq& 2f(y_0) - f(z_0) -f(x_0) + 2 \delta,
\end{eqnarray}
where we used the fact that $(f^\epsilon)^\epsilon(x_0)=f^{2\e}(x_0)$ (since $X$ is a length space), and that by definition $(f^\epsilon)_\epsilon(x_0)\ge f(x_0)$.

Note that if $d_X(w,x_0)=2\e$ then $$f(w)\le f^{2\e}(x_0)=f(x_0)+\frac{f^{2\e}(x_0) - f(x_0)}{2 \epsilon} d_X(w, x_0).$$ Hence for all $w \in \partial\big(B(x_0,2\e) \setminus \{x_0\}\big)$ we have
\begin{equation}\label{eq;simple}
f(w) \leq f(x_0) + \frac{f^{2\e}(x_0)  - f(x_0)}{2 \epsilon} d_X(w, x_0).
\end{equation}
Since $f^{2\e}(x_0)  - f(x_0)\ge 0$, we may apply the fact that $f$ satisfies comparison with distance functions from above to deduce that~\eqref{eq;simple} holds for every $w \in B(x_0,2\epsilon) \setminus \{x_0\}$, and thus for every $w \in B(x_0,2\epsilon)$.  Substituting $w = y_0$, we see that
\begin{eqnarray}\label{eq:for combination}
\nonumber2f(y_0)-f(x_0)-f(z_0)&\le& f(x_0)-f(z_0)+ \frac{f^{2\e}(x_0)  - f(x_0)}{\epsilon} d_X(y_0, x_0)\\&\le&
-\left(1-\frac{d_X(y_0,x_0)}{\e}\right)\left(f^{2\e}(x_0)  - f(x_0)\right)+\delta\nonumber\\&\le& \delta,
\end{eqnarray}
where we used the fact that $d_X(y_0,x_0)\le \e$. Since~\eqref{eq:for combination} holds for all $\delta>0$, the required result follows from a combination of~\eqref{eq:use def} and~\eqref{eq:for combination}.
\end{proof}

\begin{lemma} \label{l.coveringreduction}
Assume the following structures and definitions:
\begin{enumerate}
\item A length space $X$, a closed $Y\subseteq X$, a metric tree $T$, and a Lipschitz $f:Y \to T$.
\item A fixed $x_0 \in X \setminus Y$ and the set $\widehat X$ defined as the space of finite-length closed paths in $X$ (parameterized at unit speed) that begin at $x_0$ and remain in $X \setminus Y$ except possibly at right endpoints.
\item A metric on $\widehat X$ defined as follows: given $\widehat u, \widehat v \in \widehat X$, $d_{\widehat X}\left(\widehat u, \widehat v\right)$ is the sum of the lengths of the portions of the two paths that occur {\em after} the largest time at which $\widehat u$ and $\widehat v$ agree.  (Note that $\widehat X$ is an $\mathbb R$-tree under this metric.)
\item The ``covering map'' $M:\widehat X \to X$ that sends a path in $\widehat X$ to its right endpoint.
\end{enumerate}
If $\widetilde f:X\to T$ is an AMLE of $f$, then $\widehat f\eqdef
\widetilde f \circ M:\widehat X\to T$ is an AMLE of $f \circ M$
(which is defined on $\widehat Y\eqdef M^{-1}(Y)$). \end{lemma}
\begin{proof}First we claim that $M$ is path-length preserving.  That is, if
$\widehat \gamma$ is any rectifiable path in $\widehat X$ then
$\gamma \eqdef M \circ \widehat \gamma$ is a rectifiable path of the
same length in $X$. This is true by definition if $L \circ \widehat
\gamma$ (here $L(\cdot)$ denotes path length) is strictly
increasing, and similarly if $L \circ \widehat \gamma$ is strictly
decreasing.  Since the length of $\widehat \gamma$ is the total
variation of $L \circ \widehat \gamma$ (and the latter is finite),
the general statement can be derived by approximating $\widehat
\gamma$ with paths for which $L \circ \widehat \gamma$ is piecewise
monotone. To do this, first note that one can take an increasing set
of times $0=t_0, t_1, \ldots, t_k$ such that the total variation of
$L \circ \widehat \gamma$ restricted to those times is arbitrarily
close to the unrestricted total variation. Then the length of
$\gamma$ traversed between times $t_j$ and $t_{j+1}$ is at least
$r\eqdef|L \circ \widehat \gamma(t_j) - L \circ \widehat
\gamma(t_{j+1})|$
--- this is because the longer of the two paths $\widehat \gamma(t_j)$
and $\widehat \gamma(t_{j+1})$ contains a segment of length at least
$r$ that is not part of the other, and $\gamma$ must traverse all
the points of that segment in order (or in reverse order) somewhere
between times $t_j$ and $t_{j+1}$.

Consider an open subset $\widehat W \subseteq \widehat X \setminus
\widehat Y$ and note that $W\eqdef M\left(\widehat W\right)$ is also
open.  We need to show that if $\widetilde f$ is an AMLE of $f$ then
we cannot have $\Lip_{\widehat W}\left(\widehat f\right)
> \Lip_{\partial \widehat W}\left(\widehat f\right)$.

Indeed, suppose we had $\Lip_{\widehat W}\left(\widehat f\right)
> \Lip_{\partial \widehat W}\left(\widehat f\right)$. Then we could find a path $\widehat \gamma$
within $\widehat W$ connecting points $a,b\in \widehat W$ such that
$$\frac{d_{\widehat X}\left(\widehat f(a), \widehat f(b)\right)}{L(\gamma)} > m$$ for some
$$m > \Lip_{\partial \widehat W}\left(\widehat f\right).$$ Since $d_{\widehat W}\left(\widehat
f(a), \widehat f(\widehat\gamma(s))\right)$ is Lipschitz (hence a.e.\
differentiable) in $s$, we can find an $s_0$ at which its derivative
is greater than $m$.   Thus, for all sufficiently small
$\epsilon_0$, writing $x_0=\gamma(s_0)$ (where $\gamma=M\circ
\widehat \gamma$), we can find points $x_1,x_{-1}\in  W$ such that
 $d_{X}(x_1,x_0)=d_{ X}(x_{-1},x_0)=\e_0$ and $\widetilde
f(x_1), \widetilde f(x_{-1})$ are both at distance greater than $m
\epsilon_0$ from $\widetilde f(x_0)$ and lie in distinct components
of $T \setminus \left\{\widetilde f(x_0)\right\}$.

Now consider some $\epsilon>0$ much smaller than $\epsilon_0$. Since
$\widetilde f$ is an AMLE of $f$, the $T$-comparison property
implies that the function $$g(\cdot) \eqdef d_T\left(\widetilde f(\cdot),
\widetilde f(x_0)\right) $$ satisfies comparison with distance
functions from above.    Moreover, along any near-geodesic from
$x_0$ to $x_1$, the function $g$ increases at an average speed greater than $m$.   Write, as before, $g^\epsilon(x)\eqdef\sup_{B(x,\e)}g$. Let $C$ be the Lipschitz constant of $\widetilde f$, so that
 $|g - g^\epsilon| \le  C\epsilon$.  Note that $g(x_0)=0$ and $g(x_1)>m\e_0$, and therefore $g^\e(x_1)-g^\e(x_0)>m\e_0-2C\e$. By considering a near-geodesic from $x_0$ to $x_1$, this implies that when $\epsilon$ is small enough, we can find points $y_1,y_2\in B(x_0,\e_0)$ with $g^\epsilon(y_2)-g^\epsilon( y_1) >
  m\e$ and $d_T(y_1, y_2) \le \epsilon$, such that $\widetilde f(y_1)$ and $\widetilde f(y_2)$ lie in the same
component of  $T \setminus \left\{\widetilde f(x_0)\right\}$ as $\widetilde f(x_1)$, and both $\widetilde f(y_1)$ and $\widetilde f(y_2)$ are at distance at least $3C\e$ from $\widetilde f(x_0)$.

Fix $\delta>0$.  Applying Lemma~\ref{l.armstrongsmart} inductively
we obtain a sequence of points $\{y_i\}_{i=1}^k$  such that for all
$i\in \{1,\ldots,k-1\}$ we have $d_X(y_i,y_{i+1})\le \e$ and for all
$i\in \{1,\ldots,k-2\}$,
\begin{equation}\label{eq:increase}
g^\e(y_{i+2})-g^\e(y_{i+1})\ge g^\e(y_{i+1})-g^\e(y_i)-\frac{\delta}{2^i}.
\end{equation}
This iterative construction can continue until the first $k$ for which $y_k$ has distance at most $\e$ from $\partial W$.
It follows from~\eqref{eq:increase} that
$$
g^\e(y_{i+1})-g^\e(y_i)\ge g^\epsilon(y_2)-g^\epsilon( y_1)-\sum_{j=1}^\infty \frac{\delta}{2^j}>m\e-\delta.
$$
Thus, assuming $\delta$ is small enough, we have $g^\e(y_{i+1})-g^\e(y_i)>m\e$ for all $i\ge 1$. It follows that for all $j>i\ge 1$ we have
\begin{equation}\label{eq:too big}
g^\e(y_j)>g^\e(y_i)+(j-i)m\e.
\end{equation}
A consequence of~\eqref{eq:too big} is that, since $g^\e$ is bounded, the above construction
cannot continue indefinitely without reaching a point within $\epsilon$ distance from $\partial W$.

Another consequence of~\eqref{eq:too big} and the fact that $\widetilde f$ is $C$-Lipschitz is that for all $j>i$,
 \begin{equation}\label{eq:17}
 g(y_j)>g(y_i)+(j-i)m\e-2C\e.
  \end{equation} In particular, since $g(y_1)\ge 3C\e$, it follows from~\eqref{eq:17} that for all $i\ge 1$ we have $g(y_i)\ge C\e$. This implies that the points $\left\{\widetilde f(y_i)\right\}_{i=1}^\infty$ are all in the same
component of  $T \setminus \left\{\widetilde f(x_0)\right\}$ as $\widetilde f(x_1)$, since otherwise if $i$ is the first index such that $\widetilde f(y_i)$ is not in this component, then $d_T\left(\widetilde f(y_{i}),\widetilde f(y_{i-1})\right)\ge 2C\e$, contradicting the fact that $\widetilde f$ is $C$-Lipschitz.

 We similarly construct the sequence
$\{z_i\}_{i=1}^\ell$, starting with the point $x_{-1}$ instead of the point $x_{1}$, such that
for all $j>i$,
 \begin{equation*}\label{eq:18}
 g(z_j)>g(z_i)+(j-i)m\e-2C\e.
  \end{equation*}
As before, the entire sequence $\left\{\widetilde f(z_i )\right\}_{i=1}^\ell$ must remain in the
same component of  $T \setminus \left\{\widetilde f(x_0)\right\}$ as $f(x_{-1})$ and $z_\ell$ is within $\epsilon$ distance from $\partial W$.

Now consider some $M$ pre-image $\widehat x_0$ of $x_0$ that is
contained in $\widehat W$, and let $D$ denote the the distance from
$\widehat x_0$ to $\partial \widehat W$.  Among the rectifiable paths in $W$ from one boundary point of $W$ to another
that pass through all the $z_k$ in reverse order and the
subsequently the $y_k$ in order, let $\gamma_0$ be one which is near the shortest.   Take any path
$\widehat \gamma_0$ through $\widehat x_0$ such that $M \circ
\widehat \gamma_0=\gamma_0$. Then consider a maximal arc of
this path contained in $\widehat W$ and containing $\widehat x_0$ (which necessarily has length at
least $D$ and connects two points on $\partial \widehat W$).  The
length of this maximal arc is at least $D$ and the change in
$\widehat f$ from one endpoint of the arc to the other is (in
distance) at least $m$ times the length of the arc, plus an
$O(\epsilon_0)$ error, which contradicts the definition of $m$.
\end{proof}

\begin{lemma} \label{l.uniquenesshalf}
 Given a length space $(X,d_X)$, a closed $Y\subseteq X$, a metric tree $(T,d_T)$, and a
 Lipschitz function
 $f:Y \to T$, there exists at most one $\widetilde f: X\to T$ which
 is an AMLE of $f$.
\end{lemma}

\begin{proof}By Proposition~\ref{p.treecomparison} we must show that if $g,h:X\to T$ are continuous functions
such that $g=h$ on $Y$, and $g,h$ both satisfy $T$-comparison on $X
\setminus Y$, then $g = h$ throughout $X$. By Lemma \ref{l.coveringreduction}, it is enough to prove this in the case that $X$ is an $\mathbb R$-tree: in particular, we may assume that $X$ and $X \setminus Y$ are simply connected (note that the connected components of an open subset of a metric tree are simply connected).

For the sake of obtaining
a contradiction, suppose that $g,h$ satisfy these hypotheses, but
$g(x) \not = h(x)$ for some $x \in X \setminus Y$. Then the
hypotheses still hold if we replace $Y$ with the complement of the
connected component of $\{y\in X:\  g(y) \not = h(y) \}$ containing
$x$. In other words, we lose no generality in assuming $g(x) \not =
h(x)$ for all $x \in X \setminus Y$.  The pair $(g(\cdot),h(\cdot))$
may then be viewed as a map from $X \setminus Y$ to $$\mathcal T
\eqdef \{(t_1,t_2) \in T \times T :\  t_1 \not = t_2 \}.$$ We will
use this map to define a certain pair of real-valued functions on
$X$.

To this end, consider an arbitrary continuous function $P: [0,1] \to
\mathcal T$. Define functions $t_1,t_2:[0,1] \to T$ by writing
$(t_1(s), t_2(s)) = P(s)$.  Let $I((t_1,t_2))$ denote the geodesic
joining $t_1$ and $t_2$ in $T$.  For every $s \in [0,1]$, we will
define an isometry $\Psi^P_s$ from $I(P(s))$ to an interval
$(a_1(s), a_2(s))$ of $\mathbb R$, sending $t_1$ to $a_1(s)$ and
$t_2$ to $a_2(s)$. Clearly, the values of $a_1(s)$ and $a_2(s)$
determine the isometry, and for each $s$, we must have $$a_2(s) -
a_1(s) = d_T(t_1(s), t_2(s)).$$  However, the latter observation
only determines $a_1(s)$ and $a_2(s)$ up to the addition of a single
constant to both values. This constant is determined by the
following requirements:
 \begin{enumerate}
 \item $a_1(0) = 0$.
 \item For each fixed $t\in T$, the function $s\mapsto \Psi^P_s(t)$ is constant on every connected interval of the (open)
 set $\{s\in [0,1]:\  t \in I(P(s)) \}$.
 \end{enumerate}
Informally, at each time $s$, the geodesic $I(P(s))$ is ``glued''
isometrically to the interval $(a_1(s),a_2(s)) \subseteq \R$.  As
$s$ increases, if $t_1$ and $t_2$ move closer to each other, then
points are being removed from the ends of $I(P(s))$, and these
points are ``unglued'' from $\R$.  As $t_1$ and $t_2$ move further
from each other, new points are added to the geodesic and these new
points are glued back onto $\R$.

 When $P'$ and $P$ are paths in
$\mathcal T$ as above, write $P\sim P'$ if $P(0) = P'(0)$ and $P(1)
= P'(1)$ and for each $s \in [0,1]$ we have $I(P(s)) \cap I(P'(s))
\not = \emptyset.$ We claim that in this case $\Psi^P_1 =
\Psi^{P'}_1$. Note that given $s$ and $t \in I(P(s)) \cap I(P'(s))$,
 requirement $(2)$ above implies that $\Psi^P_s$ and $\Psi^{P'}_s$
agree up to additive constant on a neighborhood of $s$; namely, they
must agree on the connected component of $\{s':\  t \in I(P(s)) \cap
I(P'(s)) \}$ containing $s$.  Since $\Psi^P_s$ and $\Psi^{P'}_s$
agree up to additive constant on an open neighborhood of every $s
\in [0,1]$, they must be equal up to additive constant throughout
the interval, and requirement $(1)$ above implies that this constant
is zero.

A corollary of this discussion is that $\Psi^P_1 = \Psi^{P'}_1$
whenever $P$ and $P'$ are homotopically equivalent paths in
$\mathcal T$.  To obtain this, it is enough to observe that if $P^r$
is a homotopy, with $r \in [0,1]$ and $P=P^0, P'=P^1$, then for each
$r \in [0,1]$ we have $P^r \sim P^{r'}$ for all $r'$ in some
neighborhood of $r$. This implies that $\Psi^{P^r}_1$ (as a function
of $r$) is constant on a neighborhood of each point in $[0,1]$,
hence constant throughout $[0,1]$.

Since we are assuming $X \setminus Y$ is simply connected, we can
fix a point $x_0 \in X \setminus Y$ and define the pair $(a_1(x),
a_2(x))$ to be the value $(a_1(1), a_2(1))$ obtained above by taking
$P(s) = (g(p(s)), h(p(s)))$ where $p:[0,1] \to X$ is any path from
$x_0$ to $x$.  For $y$ in some neighborhood of each $x \in X
\setminus Y$, and for some $t \in T$, we have that $a_1(y)$ is an
affine function of $d_T(t, f(y))$.  For this, it suffices to take a
neighborhood and a $t$ such that $t \in I(P(s))$ throughout that
neighborhood and use requirement $(2)$ above. In this neighborhood,
$T$-comparison implies that $a_1$ satisfies comparison with distance
functions from below and $a_2$ satisfies comparison with distance
functions from above.

%This local
%result implies a global result:  $a_1$ satisfies comparison with
%distance functions from below and $a_2$ satisfies comparison with
%distance functions from above on all of $X \setminus Y$.

Now let $C$ be the Lipschitz constant of $f$, and for $\delta>0$
write $$ U_\delta \eqdef \left\{ x:\  a_1(x) < a_2(x) + 2 C \delta
\right\}.
 $$
We claim that the argument in the paragraph just above implies that
$a_2$ satisfies comparison with distance functions from above on
$B(x,\delta)$ for any $x \in U_\delta$.  To see this, first observe
that the fact that $a_2$ is Lipschitz with constant at most $C$
implies that $B(x,\delta) \subseteq X \setminus Y$, and then take
$t$ to be the midpoint in $T$ between $g(x)$ and $h(x)$, noting that
since $g$ and $h$ are both Lipschitz with constant $C$, we have that
$t$ is on the geodesic between $g(x')$ and $h(x')$ for all $x' \in
B(x,\delta)$.

We may now apply Lemma \ref{l.compabovelocal}  (where the $Y$ of the
lemma statement is chosen so that $U_\delta = X \setminus Y$) to see
that $a_2 \leq \widetilde a_2$ on $U_\delta$, where $\widetilde a_2$
is the AMLE of the restriction of $a_2$ to $\partial U_\delta$. By
symmetry, we may apply the same arguments to $-a_1$ to obtain that
$a_1 \geq \widetilde a_1$ on $U_\delta$.  Since $a_1 < a_2$ by
construction, we now have

\begin{equation} \label{e.sandwiching} \widetilde a_1 \leq a_1 < a_2 \leq \widetilde a_2.\end{equation}

Now it is a standard fact (see~\cite{ACJ04}) about AMLE that the
suprema and infima of a difference of AMLEs (in this case
$\widetilde a_2 - \widetilde a_1$) is obtained on the boundary set
(in this case $\partial U_\delta$). This implies that $\widetilde
a_1 \leq \widetilde a_2 \leq \widetilde a_1 + 2C\delta$ throughout
the set $U_\delta$.  By \eqref{e.sandwiching}, we now have that
$$\sup_{x \in U_\delta} |a_1(x) - a_2(x)| \leq 2C\delta$$  This
implies $\sup_{x \in X} |a_1(x) - a_2(x)| \leq 2C \delta$, and since
this holds for all $\delta > 0$, we have $a_1 = a_2$ throughout $X$,
a contradiction.
\end{proof}

\section{Proofs of Theorem~\ref{thm:local-global} and Theorem~\ref{thm:aronszajn}} \label{s.infharmAMLE}

We first present the simple proof of Theorem~\ref{thm:local-global},
i.e, the local-global result for tree-valued $\infty$-harmonic
functions. The proof is a modification of an argument
in~\cite{PSSW09} from the setting of real-valued mappings to the
setting of tree-valued mappings.

\begin{proof}[Proof of Theorem~\ref{thm:local-global}] Let $U\subseteq G\setminus \Omega$ be an open subset of $G$. Denote
\begin{equation}\label{eq:def lip const closure}
L = \mathrm{Lip}_{\overline U} (f) = \max_{\substack{x,y \in \overline{U}\\x\neq y}} \frac{d_T(f(x),f(y))}{d_G(x,y)}.
\end{equation}
Let $x,y\in \overline U$ be points at which the maximum in~\eqref{eq:def lip const closure} is attained and $d_G(x,y)$ is maximal among all such points. We will be done if we show that $x,y\in \partial U$. Assume for the sake of contradiction that $x\in U$ (the case $y \in U$ being similar).

If $x$ is in the interior of an edge of $G$, we could move $x$ slightly along the edge and increase $d_G(x,y)$ without decreasing $d_T(f(x),f(y))/d_G(x,y)$. So, assume that $x\in V$. The fact that $f$ is $\infty$-harmonic on $V\setminus \Omega\supseteq \{x\}$ means that there exist  $u,v\in N_G(x)$ with $d_T(f(u),f(v))=2L$ and $f(x)$ is a midpoint between $f(u)$ and $f(v)$ in $T$. Since $x\in U$ (and $f$ is linear on the edges of $G$) there exists $\e>0$ and $z_u,z_v\in U$ such that $z_u\in xu\in E$, $z_v\in xv\in E$, $d_G(x,z_u)=d_G(x,z_v)=\e$ and $f(z_u)$ (resp. $f(z_v)$) is the point on the geodesic in $T$ joining $f(x)$ and $f(u)$  (resp. $f(v)$) at distance $L\e$ from $f(x)$. Because $T$ is a metric tree, either $d_T(f(x),f(z_u))=d_T(f(x),f(y))+L\e$ or $d_T(f(x),f(z_v))=d_T(f(x),f(y))+L\e$. Assume without loss of generality that $d_T(f(x),f(z_u))=d_T(f(x),f(y))+L\e=L(d_G(x,y)+\e)$. Then $d_G(z_u,y)=d_G(x,y)+\e$ and $d_T(f(z_u),f(y))/d_G(z_u,y)=L$, contradicting the maximality of $d_G(x,y)$.
\end{proof}

 As in the classical proof that metric convexity and the binary
intersection property implies the isometric extension property (see
~\cite[Prop.~1.4]{BL00}), for the proof of
Theorem~\ref{thm:aronszajn} we shall construct $\widetilde f$ by
extending to one additional point at a time. The proof of
Theorem~\ref{thm:aronszajn} relies on a specific choice of the
ordering of the points for the purpose of such a point-by-point
construction. Our argument uses a variant of an algorithm from~\cite{LLPSU99}.

\begin{proof}[Proof of Theorem~\ref{thm:aronszajn}] Write $|V\setminus \Omega|=n$. We shall construct inductively a special ordering $w_1,\ldots,w_n$ of the points of $V\setminus \Omega$, and extend $f$ to these points one by one according to this ordering. Assume that  $w_1,\ldots,w_k$ have been defined, as well as the values $\widetilde f(w_1),\ldots,\widetilde f(w_k)\in Z$ (if $k=0$ this assumption is vacuous).

Write $\Omega_0=\Omega$ and $\Omega_k=\Omega\cup\{w_1,\ldots,w_k\}$. Given distinct $x,y\in V$ we shall say that $x_0,x_1,\ldots x_\ell\in V$ is a path joining $x$ and $y$ which is external to $\Omega_k$ if $x_0=x$, $x_\ell=y$, and for all $i\in\{0,\ldots,\ell-1\}$ we have $x_ix_{i+1}\in E$ and $\{x_i,x_{i+1}\}\not \subseteq \Omega_k$. Let $d_k(x,y)$ be the minimum over $\ell\in \N$ such that there exists a path $x_0,x_1,\ldots x_\ell\in V$ joining $x$ and $y$ which is external to $\Omega_k$. If no such path exists we set $d_k(x,y)=\infty$. We also set $d_k(x,x)=0$ for all $x\in V$. Then $d_k:V\times V \to \{0\}\cup\N\cup\{\infty\}$ clearly satisfies the triangle inequality and $d_k(\cdot,\cdot)\ge d_G(\cdot,\cdot)$ pointwise.

We distinguish between two cases:

%\medskip

\noindent{\bf Case 1.} For all distinct $x,y\in \Omega_k$ we have $d_k(x,y)=\infty$. In this case order the points of $V\setminus \Omega_k$ arbitrarily, i.e., $V\setminus \Omega_k=\{w_{k+1},\ldots,w_n\}$. If $w\in \{w_{k+1},\ldots,w_n\}$ then by the connectedness of $G$, there exists a path in $G$ joining $w$ and some point $x_w\in \Omega_k$. Note that $x_w$ is uniquely determined by $w$, since if there were another path joining $w$ and some point $y_w\in \Omega_k$ which isn't $x_w$ then $d_k(x_w,y_w)<\infty$, contradicting our assumption in Case 1. We can therefore define in this case $\widetilde f(w)=\widetilde f (x_w)$.

\medskip

\noindent{\bf Case 2.} for some distinct $x,y\in \Omega_k$ we have $d_k(x,y)<\infty$. In this case define
\begin{equation}\label{eq:def Lk}
L_k=\max_{\substack{x,y\in \Omega_k\\x\neq y}}\frac{d_Z\left(\widetilde f(x),\widetilde f(y)\right)}{d_k(x,y)}.
\end{equation}
Our assumption implies that $L_k>0$. Choose $x,y\in \Omega_k$ that are distinct and satisfy $L_kd_k(x,y)=d_Z\left(\widetilde f(x),\widetilde f(y)\right)$. Write $\ell=d_k(x,y)$ and let $x_0,x_1,\ldots x_\ell\in V$ be a path joining $x$ and $y$ which is external to $\Omega_k$. Then $x_1\notin \Omega_k$, so we may define $w_{k+1}=x_1$. We claim that
\begin{equation}\label{eq:intersection}
\bigcap_{a\in \Omega_k}B_Z\left(\widetilde f(a),L_kd_k(a,w_{k+1})\right)\neq \emptyset.
\end{equation}
To prove~\eqref{eq:intersection}, by the fact that $Z$ has the binary intersection property, it suffices to show that for all $a,b\in \Omega_k$ we have
\begin{equation}\label{eq:2 intersection}
B_Z\left(\widetilde f(a),L_kd_k(a,w_{k+1})\right)\cap B_Z\left(\widetilde f(b),L_kd_k(b,w_{k+1})\right)\neq \emptyset.
\end{equation}
If either $d_k\left(a,w_{k+1}\right)=\infty$ or $d_k\left(b,w_{k+1}\right)=\infty$ then~\eqref{eq:2 intersection} is trivial. Assume therefore that  $d_k\left(a,w_{k+1}\right)$ and $d_k\left(b,w_{k+1}\right)$ are finite. Define $\lambda\in [0,1]$ by
\begin{equation}\label{eq;def lambda}
\lambda=\frac{d_k\left(a,w_{k+1}\right)}{d_k\left(a,w_{k+1}\right)+d_k\left(a,w_{k+1}\right)}.
\end{equation}
Since $Z$ is metrically convex, there exists a point $z\in Z$ such that
\begin{equation}\label{eq:mc}
d_Z\left(z,\widetilde f(a)\right)=\lambda d_Z\left(\widetilde f(a),\widetilde f(b)\right)\quad\mathrm{and}\quad d_Z\left(z,\widetilde f(b)\right)=(1-\lambda) d_Z\left(\widetilde f(a),\widetilde f(b)\right).
\end{equation}
The definition of $L_k$ implies \begin{equation}\label{eq:use L_k}
d_Z\left(\widetilde f(a),\widetilde f(b)\right)\le L_kd_k(a,b)\le L_k\big(d_k(a,w_{k+1})+d_k(b,w_{k+1})\big).
 \end{equation} Using~\eqref{eq;def lambda} and~\eqref{eq:mc}, we deduce from~\eqref{eq:use L_k} that $$d_Z\left(z,\widetilde f(a)\right)\le L_kd_k(a,w_{k+1})\quad\mathrm{and}\quad d_Z\left(z,\widetilde f(b)\right)\le L_kd_k(b,w_{k+1}),
  $$proving~\eqref{eq:2 intersection}. Having proved~\eqref{eq:intersection}, we let $\widetilde f(w_{k+1})$ be an arbitrary point satisfying
  \begin{equation}\label{eq:in intersection}
  \widetilde f(w_{k+1})\in \bigcap_{a\in \Omega_k}B_Z\left(\widetilde f(a),L_kd_k(a,w_{k+1})\right).
  \end{equation}

  The above inductive construction produces a function $\widetilde f:V\to Z$ that extends $f$. We claim that $\widetilde f$ is $\infty$-harmonic on $V\setminus \Omega$. To see this note that for all $x,y\in V$ the sequence $\{d_k(x,y)\}_{k=1}^n\subseteq \{0\}\cup\N\cup\{\infty\}$ is non-decreasing. We shall next show that the sequence $\{L_k\}_{k=1}^n$, defined in~\eqref{eq:def Lk}, is non-increasing. Indeed, assume that $L_{k+1}>0$ and take distinct $a,b\in \Omega_{k+1}$ such that $L_{k+1}d_{k+1}(a,b)=d_Z\left(\widetilde f(a),\widetilde f(b)\right)$. If $a,b\in \Omega_k$ then it follows from the definition of $L_k$ that $L_{k+1}\le L_k$, since $d_{k+1}(a,b)\ge d_k(a,b)$. By symmetry, it remains to deal with the case $a\in \Omega_k$ and $b=w_{k=1}$. In this case, since by our construction we have $\widetilde f(w_{k+1})\in B_Z\left(\widetilde f(a),L_kd_k(a,w_{k+1})\right)\subseteq B_Z\left(\widetilde f(a),L_{k}d_{k+1}(a,w_{k+1})\right)$, it follows once more that $L_{k+1}\le L_k$.

Fix $k\in \{0,\ldots,n-1\}$. If $\widetilde f(w_{k+1})$ was defined in Case 1 of our inductive construction, then $\widetilde f$ is constant on $N_G(w_{k+1})\cup\{w_{k+1}\}$, in which fact the $\infty$-harmonic conditions~\eqref{eq:first infty}, \eqref{eq:second infty} for $\widetilde f$ at $w_{k+1}$ hold trivially. If, on the other hand, $\widetilde f(w_{k+1})$ was defined in Case 2 of our inductive construction, then there exist distinct $x,y\in \Omega_k$ with $L_kd_k(x,y)=d_Z\left(\widetilde f(x),\widetilde f(y)\right)$, such that for $\ell=d_k(x,y)$ there are $x_0,x_1,\ldots x_\ell\in V$ which form a path joining $x$ and $y$ which is external to $\Omega_k$, and $x_1=w_{k+1}$. For every $i\in \{1,\ldots,\ell-1\}$ either $x_i\notin \Omega_k$ or $x_{i+1}\notin \Omega_{k}$, and therefore at least one of the values $\widetilde f(x_i),\widetilde f(x_{i+1})$ was define after stage $k+1$ of our inductive construction. This means that for some $j\ge k$ we have $|\Omega_j\cap\{x_i,x_{i+1}\}|=1$ and
\begin{equation}\label{eq:increment}
d_Z\left(\widetilde f(x_i),\widetilde f(x_{i+1})\right)\le L_jd_j(x_i,x_{i+1})=L_j\le L_k,
\end{equation}
where we used the fact that $x_ix_{i+1}\in E$, and therefore, since $|\Omega_j\cap\{x_i,x_{i+1}\}|=1$, the path $x_i,x_{i+1}$ is external to $\Omega_j$. Thus
\begin{eqnarray}\label{eq:constant increments}
\nonumber L_k\ell&=&d_Z\left(\widetilde f(x),\widetilde f(y)\right)\\\nonumber&\le& d_Z\left(\widetilde f(x_0),\widetilde f(x_2)\right)+ d_Z\left(\widetilde f(x_2),\widetilde f(x_\ell)\right)\\\nonumber &\le& d_Z\left(\widetilde f(x_0),\widetilde f(x_2)\right)+ \sum_{i=2}^{\ell-1}d_Z\left(\widetilde f(x_i),\widetilde f(x_{i+1})\right)\\&\stackrel{\eqref{eq:increment}}{\le}&  d_Z\left(\widetilde f(x_0),\widetilde f(x_1)\right)+ d_Z\left(\widetilde f(x_1),\widetilde f(x_2)\right)+ L_k(\ell-2)\nonumber\\&\stackrel{\eqref{eq:increment}}{\le}& L_k\ell.
\end{eqnarray}
It follows that all the inequalities in~\eqref{eq:constant increments} actually hold as equality. Therefore we have $d_Z\left(\widetilde f(x),\widetilde f(w_{k+1})\right)= d_Z\left(\widetilde f(w_{k+1}),\widetilde f(x_{2})\right)=L_k$ and $d_Z\left(\widetilde f(x),\widetilde f(x_{2})\right)=2L_k$. Since by construction $x,x_2\in N_G(w_{k+1})$, in order to show that $\widetilde f$ is $\infty$-harmonic at $w_{k+1}$ it remains to check that for all $u\in N_G(w_{k+1})$ we have $d_Z\left(\widetilde f(u),\widetilde f(w_{k+1})\right)\le L_k$. But, our construction ensures that for some $j\ge k$ we have $d_Z\left(\widetilde f(u),\widetilde f(w_{k+1})\right)\le L_jd_j(u,w_{k+1})=L_j$ (using $uw_{k+1}\in E$), and the required result follows since $L_j\le L_k$.
\end{proof}

\section{Existence} \label{s.existence}

Here we prove the existence part of Theorem~\ref{thm:tree main},
i.e., we establish the following:

\begin{theorem}\label{thm:exist ext}
Let $(X,d_X)$ be a locally compact length space and $(T,d_T)$ a
metric tree. Then for every closed $Y\subseteq X$, every Lipschitz
mapping $f:Y\to T$ admits an AMLE.
\end{theorem}

\begin{proof} Assume first that $X$ is compact. We will construct an AMLE $\widetilde f$ of $f$ as a limit of discrete approximations.

For each $\e \in (0,1/4)$, let $\Lambda_\e$ be a finite subset of $X$ such that
\begin{equation}\label{eq:lambda}
X \subseteq \bigcup_{x \in \Lambda_\e} B_X(x,\e)\quad \mathrm{and}\quad Y \subseteq \bigcup_{y \in \Lambda_\e \cap Y} B_X(y,\e).
\end{equation}
Let $G_\e$ be the graph whose vertices are the elements of $\Lambda_\e$, with $x,y\in\Lambda_\e$ adjacent when $d_X(x,y) \leq \sqrt{\e}$.

For any $x$ and $y$ in $\Lambda_\e$, we can find an arbitrarily-close-to-minimal length path between them and a sequence of points $x=x_0, x_1, x_2, \ldots, x_k = y$ spaced at intervals of $\sqrt{\e} - 2 \e$ along the path, where $k-1$ is the integer part of $d_X(x,y)/(\sqrt{\e} - 2 \e)$, and can then find points $\tilde x_i \in B(x_i,\e) \cap \Lambda_\e$.  Since $d(x_i,x_{i+1}) \leq \sqrt{\e}$ we conclude that $d_{G_\e}(x,y) \leq k$.  It is also clear that $d_{G_\e}(x,y) \geq d(x,y)/\sqrt{\e}$.
Hence,
\begin{equation}\label{e.metricapprox} \left| d_{G_\e}(x,y)\sqrt{\e} - d_X(x,y) \right| \leq C\sqrt{\e},\end{equation} where $C$ depends only on the diameter of $X$.

Let $\widetilde f_\e$ be an $\infty$-harmonic extension of $f|_{Y\cap \Lambda_\e}$  to all of $G_\e$, the existence of which is due to Theorem~\ref{thm:aronszajn} (since $T$ is a $1$-absolute Lipschitz retract).  Note that on $\Lambda_\e$ we have the point-wise inequality $d_X(\cdot,\cdot)\le \sqrt{\e} d_{G_\e}(\cdot,\cdot)$. It follows that the Lipschitz constant of $f|_{Y\cap \Lambda_\e}$ with respect to the metric $\sqrt{\e}d_{G_\e}$ is bounded above by $\Lip_Y(f)$, and hence the Lipschitz constant of $\widetilde f_\e$ with respect to the metric $\sqrt{\e}d_{G_\e}$ is also bounded above by $\Lip_Y(f)$.

Let $\mathcal N_\e\subseteq \Lambda_\e$ be a $\sqrt{\e}$-net in $(\Lambda_\e,d_X)$, i.e., a maximal subset of $\Lambda_\e$, any two elements of which are separated in the metric $d_X$ by at least $\sqrt{\e}$. For any distinct $x,y\in \mathcal N_\e$ we have
\begin{multline}\label{eq:sqrt}
d_T\left(\widetilde f_\e(x),\widetilde f_\e(y)\right)\le \Lip_Y(f)\sqrt{\e}d_{G_\e}(x,y)\\\stackrel{\eqref{e.metricapprox}}{\le} \Lip_Y(f)\left(d_X(x,y)+C\sqrt{\e}\right)\le \Lip_Y(f)(1+C)d_X(x,y).
\end{multline}
It follows that we can extends $\left.\widetilde f\right|_{\mathcal N_\e}$ to a function $f_\e^*:X\to T$ that is Lipschitz with constant $\Lip_Y(f)(1+C)$ (this extension can be done in an arbitrary way, using the fact that $T$ is a $1$-absolute Lipschitz retract). Since the functions $f^*_\e$ are equicontinuous, the  Arzela-Ascoli Theorem~\cite[Thm.\ 6.1]{MR0464128} says that there exists a subsequence $\{\e_n\}_{n=1}^\infty\subseteq (0,1/4)$ tending to zero such that $f^*_{\e_n}$ converges uniformly to $f^*:X\to T$. We aim to show that $f^*$ is an AMLE of $f$.

%there exists a subsequence of the $u_\e$ that converges uniformly on $X$ to a limiting function $u$.  We aim to show that any such $u$ is AMLE.
%We do this by showing that $u$ satisfies tree comparison, i.e., that $v(x) := d(t, u(x))$ satisfies comparison with distance functions from above for each $t \in T$.  %Similarly, write $v_\e(x) := d(t, u_\e(x))$

By Proposition~\ref{p.treecomparison} it is enough to show that for each $t\in T$ and open $W\subseteq X\setminus Y$,  $z\in X\setminus W$, $b\geq 0$ and $c\in \mathbb R$,  we have the following:
\begin{multline}\label{eq:implication}
\forall x\in \partial W\quad  d_T\left(t, f^*(x)\right)\le bd_X(x,z)+c \\ \implies \ \forall x\in W\quad d_T\left(t, f^*(x)\right)\le bd_X(x,z)+c.\end{multline}
By uniform convergence, for every $\delta>0$ there exists $n_0\in \N$ such that for every $n\ge n_0$ if for every $x\in \partial W$ we have
\begin{equation}\label{eq:boundary assumption}
d_T\left(t, f^*(x)\right)\le bd_X(x,z)+c,
\end{equation}
then for every $x\in \partial W$ we have
\begin{equation}\label{eq:assumption for tree}
d_T\left(t, f^*_{\e_n}(x)\right)\le bd_X(x,z)+c+\delta.
\end{equation}

Assume from now on that~\eqref{eq:boundary assumption} holds for all $x\in \partial W$. Let $V_{\e_n}\subseteq \mathcal N_{\e_n}$ be the set of $u\in \mathcal N_{\e_n}\subseteq \Lambda_{\e_n}$ for which there exists $w\in  W$ such that $d_X(u,w)\le \sqrt{\e_n}$. Define $W_{\e_n}$ to be the open subset of the $1$-dimensional simplicial complex corresponding to the graph $G_{\e_n}$ consisting of the union of all the half-open intervals $[u,v)$, where $u,v\in \Lambda_{\e_n}$, $uv$ is an edge of $G_{\e_n}$ and $u\in  V_{\e_n}$. Any point $v\in \partial W_{\e_n}$ of the boundary of $W_{\e_n}$ in $G_{\e_n}$ is at $d_X$-distance greater than $\sqrt{\e_n}$ from $W$, but at $d_X$-distance at most $\sqrt{\e_n}$ from some  point of $\mathcal N_{\e_n}$ whose $d_X$-distance from $W$ is at most $\sqrt{\e_n}$. Thus
\begin{equation}\label{eq:close to boundary}
v\in \partial W_{\e_n}\implies d_X(u,\partial W)\le 2\sqrt{\e_n}.
\end{equation}
Let $z_{\e_n}$ be any one of the $d_X$-closest points of $z$ in $\mathcal N_{\e_n}$. By~\eqref{eq:lambda} and the definition of $\mathcal N_{\e_n}$ we have
\begin{equation}\label{eq:z close to boundary}
d_X(z_{\e_n},z)\le \sqrt{\e_n}+\e_n\le 2\sqrt{\e_n}.
\end{equation}
 Since $f^*_{\e_n}$ is Lipschitz with constant $\Lip_Y(f)(1+C)$, for every $v\in \partial W_{\e_n}$ we have
\begin{eqnarray}\label{perturb}
\nonumber d_T\left(t, f^*_{\e_n}(v)\right)&\stackrel{\eqref{eq:assumption for tree}\wedge \eqref{eq:close to boundary}}{\le}& bd_X(v,z)+c+\delta+2\Lip_Y(f)(1+C)\sqrt{\e_n}\nonumber\\&\stackrel{\eqref{eq:z close to boundary}}{\le}& bd_X\left(v,z_{\e_n}\right)+c+\delta+2(\Lip_Y(f)(1+C)+b)\sqrt{\e_n}\nonumber\\
&\stackrel{\eqref{e.metricapprox}}{\le}& b\sqrt{\e_n}d_{G_{\e_n}}\left(v,z_{\e_n}\right)+c+\delta+(2\Lip_Y(f)(1+C)+2b+bC)\sqrt{\e_n}\nonumber\\
&\le& b\sqrt{\e_n}d_{G_{\e_n}}\left(v,z_{\e_n}\right)+c+\delta+K\sqrt{\e_n},
\end{eqnarray}
where $K>0$ is independent of $n$. Observe that if  $z_{\e_n}\in W_{\e_n}$ then since $z\notin W$ we have $d_X(z_{\e_n},\partial W)\le 6\sqrt{\e_n}$. In this case the same argument as above shows that~\eqref{perturb} holds for $v=z_{\e_n}$ as well (with a different value of $K$). Thus, the bound~\eqref{perturb} holds for all $v\in \partial (W_{\e_n}\setminus\{z_{\e_n}\})$. By Theorem~\ref{thm:local-global} and Proposition~\ref{p.treecomparison}, it follows that for every $v\in W_{\e_n}\setminus\{z_{\e_n}\}$, and hence also for all $v\in V_{\e_n}$, we have
\begin{multline}\label{eq:from}
d_T\left(t, f^*_{\e_n}(v)\right)\le b\sqrt{\e_n}d_{G_{\e_n}}\left(v,z_{\e_n}\right)+c+\delta+K\sqrt{\e_n}\\\stackrel{\eqref{e.metricapprox}\wedge\eqref{eq:z close to boundary}}{\le} bd_{X}\left(v,z\right)+c+\delta+(K+Cb+2b)\sqrt{\e_n}.
\end{multline}
Since any point of $W$ is at $d_X$-distance at most $\e_n+\sqrt{\e_n}\le 2\sqrt{\e_n}$ from $V_{\e_n}$, and since $f_{\e_n}^*$ is Lipschitz with constant independent of $n$, we see from~\eqref{eq:from} that for some $K'>0$ independent of $n$, for all $x\in W$ we have:
\begin{equation}\label{eq:before limit}
d_T\left(t, f^*_{\e_n}(x)\right)\le bd_{X}\left(x,z\right)+c+\delta+K'\sqrt{\e_n}.
\end{equation}
Letting $n$ tend to $\infty$ in~\eqref{eq:before limit}, it follows that
\begin{equation}\label{eq:with delta}
d_T\left(t, f^*(x)\right)\le bd_{X}\left(x,z\right)+c+\delta.
\end{equation}
Since~\eqref{eq:with delta} holds for all $\delta>0$, we have proved the desired implication~\eqref{eq:implication}.

When $X$ is locally compact but not necessarily compact, the proof of Theorem~\ref{thm:exist ext} follows from a direct reduction to the compact case. Indeed, by Remark~\ref{r.boundaryw} it suffices to prove~\eqref{eq:implication} when $b>0$. In this case, since $T$ is bounded, the upper bounds in~\eqref{eq:implication} are trivial if $d_X(x,z)$ is sufficiently large. Thus, it suffices to prove~\eqref{eq:implication} for the intersection of $W$ with a large enough ball centered at $z$.
\end{proof}

\section{Politics} \label{politicalsection}

In this section we prove Proposition \ref{towvalue}.  We require
some notation (in particular, a definition of value) to make the
statement of Proposition \ref{towvalue} precise.

\begin{comment}
Consider a finite connected graph $G$ with vertex set $V$ and edge
set $E$, a metric tree $T$,
%(write $x \sim y$ when $(x,y) \in E$), a tree $T$ with
%finite diameter (viewed as a geodesic metric space),
an arbitrary fixed subset $Y \subseteq V$, and a function $f:Y \to
T$.  For each $W \subseteq V$ we define $$\Lip_W(u) = \sup_{W \ni x
\sim y \in X} d(u(x),u(y)).$$ We say that $u:X \to T$ is an AM
extension of $F$ if it agrees with $F$ on $Y$ and for each $W
\subseteq X \smallsetminus Y$ and each $v: X \to T$ that agrees with
$u$ outside of $W$, we have $\Lip_W u \leq \Lip_W v$. We say that a
function $u: X \to T$ is {\bf infinity harmonic at $v$} if $v \in X
\setminus Y$ and the point $u(v)$ is the (unique) midpoint in $T$ of
two of the limit points of the set $\{ u(w) : v \sim w \}$ whose
distance from $u(v)$ is maximal.  Since a graph may be viewed as a
length space (with each edge isometric to a unit-length interval of
$\R$) the following is essentially a corollary of Theorem
\ref{thm:tree main}.
\end{comment}

A \textbf{strategy} for a player is a way of choosing the player's next
move as a function of all previously played moves and all previous coin tosses.
It is a map from the set of partially played games to moves (or in the case of a
\textbf{random strategy}, a probability distribution on moves).
We might expect a good strategy to be Markovian, i.e., a map from the
current state to the next move, but it is useful to allow more general
strategies that take into account the history.

Given two strategies $\mathcal S_{\I}, \mathcal S_{\II}$, let
$\mathcal F(\mathcal S_{\I}, \mathcal S_{\II})$ be the expected total payoff (including the
running payoffs received) when the players adopt these strategies.
We define $\mathcal F$ to be some fixed constant $C$ if the game does not terminate with
probability one, or if this expectation does not exist.

The \textbf{value of the game for player~\I }is defined as
$\sup_{\mathcal S_\I} \inf_{\mathcal S_\II} \mathcal F(\mathcal
S_\I, \mathcal S_\II)$. The \textbf{value for player~\II }is
$\inf_{\mathcal S_{\II}}\sup_{\mathcal S_{\I}} \mathcal F(\mathcal
S_{\I},\mathcal S_{\II})$.  The game has a \textbf{value} when these
two quantities are equal.  It turns out that Politics always has a
value for any choice of initial states $x_0\in V$ and $t_0\in T$;
this is a consequence of a general theorem (since the payoff
function is a zero-sum Borel-measurable function of the infinite
sequence of moves \cite{MR1665779}; see also \cite{MR2032421} for
more on stochastic games).

%In this case, the value function can be given explicitly.

%We now prove Proposition \ref{towvalue}:

\begin{proof}[Proof of Proposition~\ref{towvalue}] First we introduce some notation:
when the game position is at $x_k$, we let $y_k$ and $z_k$ denote
two of the vertices adjacent to $x_k$ that maximize
$d_T\left(\widetilde f(x_k),\widetilde f(\cdot)\right)$, chosen so
that $\widetilde f(x_k)$ is the midpoint of $\widetilde f(y_k)$ and
$\widetilde f(z_k)$. Write for $x\in V$,
$$
\delta(x) \eqdef \sup_{y\in N_G(x)} d_T\left(\widetilde
f(x),\widetilde f(y)\right), $$ and
$$
M_k \eqdef \delta(x_k) = d_T\left(\widetilde f(x_k),\widetilde
f(y_k)\right) = d_T\left(\widetilde f(x_k),\widetilde f(z_k)\right).
$$

Using this notation, we now give a strategy for player \II that
makes $d\left(\widetilde f(x_k), t_k\right)$ plus the total payoff
thus far for Player \I a {\em supermartingale}.
%The strategy may be simply described.
Player \II always chooses $t_k$ to be the element in
$\left\{\widetilde f(y_{k-1}),\widetilde f(z_{k-1})\right\}$ on which $d_T(\cdot, o_k)$ is largest; if she
wins the coin toss, she then chooses $x_k$ to be so that $\widetilde f(x_k)$ is that element.  To
establish the supermartingale property, we must show that,
regardless of player \I's strategy, we have
\begin{equation}\label{eq:super}
\mathbb E \Big[ d_T(o_k,t_k) - d_T(o_k, t_{k-1}) \Big] \geq \mathbb
E \left[ d_T \left(\widetilde f(x_k),t_k\right) -
d_T\left(\widetilde f(x_{k-1}), t_{k-1}\right) \right].
\end{equation}
 It is not hard to see that we have deterministically
\begin{equation}\label{eq:detrministically} d_T(o_k,t_k) - d_T(o_k,
t_{k-1}) \geq d_T\left(\widetilde f(x_{k-1}),t_k\right) -
d_T\left(\widetilde f(x_{k-1}), t_{k-1}\right).
\end{equation}
 Indeed, if $o_k$ and $t_{k-1}$ are in
distinct components of $T \setminus \left\{ \widetilde f(x_{k-1})
\right\}$, then the same will be true of $o_k$ and $t_k$,
and~\eqref{eq:detrministically} holds as equality; if $o_k$ and
$t_{k-1}$ are in the same component of $T \setminus \left\{ \widetilde f(x_{k-1})
\right\}$ then $o_k$ and $t_k$ will be in
opposite components of $T \setminus \left\{ \widetilde f(x_{k-1})
\right\}$, and the left hand side minus the right hand side of~\eqref{eq:detrministically} becomes twice the distance from $\widetilde f(x_{k-1})$ of the least common ancestor of $o_k$ and and $t_{k-1}$ in the tree rooted at $\widetilde f(x_{k-1})$.

Due to~\eqref{eq:detrministically}, in order to prove~\eqref{eq:super} it is
enough to show that
$$\mathbb E \left[ d_T\left(\widetilde f(x_k),t_k\right) - d_T\left(\widetilde f(x_{k-1}), t_k\right) \right] \leq 0,$$
which is clear since if player \II wins the coin toss this quantity will be $-M_k$
and if player \I wins the coin toss it will be at most $M_k$.

Next we give a very similar strategy for player \I that makes $d_T\left(\widetilde f(x_k), t_k\right)$
plus the total payoff thus far for Player \I a {\em submartingale}.
In this strategy, Player \I always chooses
$o_k$ to be the element in $\left\{\widetilde f(y_{k-1}),\widetilde f(z_{k-1})\right\}$ on which $d_T(\cdot, t_{k-1})$ is largest; if he wins the coin toss,
he then chooses $x_k$ to be so that $\widetilde f(x_k)$ is that element.  To establish the submartingale property, we must now show that
\begin{equation}\label{eq:for plugging}
\mathbb E \Big[ d_T(o_k,t_k) - d_T(o_k, t_{k-1}) \Big] \leq \mathbb E \left[ d_T\left(\widetilde f(x_k),t_k\right) - d_T\left(\widetilde f(x_{k-1}), t_{k-1}\right) \right].\end{equation}
Note that by strategy definition $o_k$ is on the opposite side of $\widetilde f(x_{k-1})$ from $t_{k-1}$, so we may write $d_T(o_k, t_{k-1}) =
M_k + d_T\left(\widetilde f(x_{k-1}), t_{k-1}\right)$.  Plugging this into~\eqref{eq:for plugging}, what we seek to show becomes
\begin{equation}\label{eq:last?}
\mathbb E \Big[ d_T(o_k,t_k) - M_k \Big] \leq \mathbb E \left[d_T\left(\widetilde f(x_k),t_k\right)\right],
\end{equation}
which we see by noting that the right hand side of~\eqref{eq:last?} is equal to $d_T(o_k, t_k)$ when player \I wins the coin toss
(and makes $\widetilde f(x_k) = o_k$) and at least $d_T(o_k,t_k) - 2M_k$ when player \II wins the coin toss, since
$d_T\left(\widetilde f(x_k),o_k\right) \leq 2M_k$ for any valid choice of $x_k$.

To conclude the proof, we need to modify the strategy in such a way that forces
the game to terminate without sacrificing the payoff expectation.  If both
players adopt the above strategy, it is clear that the increments
$d_T\left(\widetilde f(x_{k-1}), \widetilde f(x_k)\right)$ are non-decreasing, and that the distance from any
fixed endpoint of the tree has at least probability $1/2$ of increasing at each
step; from this, it follows that the length of game play is a random
variable with exponential decay.  If the other player makes other moves, which
are not optimal from the point of view of optimizing the payoff, then we can wait
until the cumulative amount the other player has ``given up'' is greater that
twice the diameter of $T$, and then force the game to end by placing a target at a single
point and subsequently always moving $x_k$ closer to that point when winning a coin toss.  (The
loss from the sub-optimality of this strategy is less than the gain from the amount the other player gave up.)
If a player adopts this strategy, then the total time duration of the game is a random variable whose law decays exponentially; this yields the uniform integrability necessary for the sub-martingale optional stopping
theorem, which implies Proposition~\ref{towvalue}.
\end{proof}

%\cite{JLPS02,Jen93,LN05,Kirsz34,Ball92,NPSS06,MP84,Mat90,PS08,hm-lips,CMS98,Obe05,Obe08,JL84,McShane34,Feff09,Jones81,Juu02,LLPSU99,ACJ04,Aro67,KN81,
%PSSW09,Isb64,BL00,AP68,SS10,MR2085299,MR0461107,MR2200122,MR2340707,MR753872,MR1379369,CHSV08,ACGR02,
%Gro53,Nao01,LN04,Sim69,Whi34,CEG01,Jen93,Haj09,MR1904557,MR2346452,MR2242966,MR2094321,MR1719573,MR1876420} \cite[pp.
%123-125]{Day73}.

%\section{Open problems}

\bibliographystyle{abbrv}

\bibliography{treeAMLE}

\end {document}